\RequirePackage{etex}

\documentclass[12pt,a4paper]{amsart}

\usepackage[a4paper, total={6in, 23.7cm}]{geometry}
\usepackage{amsmath,amsthm,amsfonts,amssymb,mathrsfs,amsbsy}
\usepackage{graphicx}
\usepackage{color}
\usepackage[noadjust]{cite}
\usepackage{marginnote}
\usepackage{xifthen}
\usepackage{mathtools}
\usepackage{dsfont}
\usepackage{bbm}
\usepackage{appendix}
\usepackage{delarray}
\usepackage{enumerate}
\usepackage{enumitem}
\usepackage{mathabx}

\usepackage{xcolor}
\usepackage{tikz}
\usepackage[hypertexnames=false]{hyperref} 

\hypersetup{
    colorlinks=true,
    linkcolor=blue,
    filecolor=magenta,      
    urlcolor=cyan,
    pdftitle={Overleaf Example},
    pdfpagemode=FullScreen,
    }

\mathtoolsset{showonlyrefs=true}

\binoppenalty=\maxdimen
\relpenalty=\maxdimen

\newenvironment{nalign}
{\begin{equation}\begin{aligned}}
		{\end{aligned}\end{equation}\ignorespacesafterend}

\newtheorem{theorem}{Theorem}
\newtheorem{corollary}[theorem]{Corollary}

\newtheorem{definition}[theorem]{Definition}

\newtheorem{lemma}[theorem]{Lemma}
\newtheorem{proposition}[theorem]{Proposition}

\numberwithin{equation}{section}
\numberwithin{theorem}{section}
\numberwithin{figure}{section}

\theoremstyle{remark}
\newtheorem{remark}[theorem]{Remark}

\newcommand\No[1]{\left\| #1 \right\|_1}
\newcommand\Ni[1]{\left\| #1 \right\|_\infty}

\newcommand\R{\mathbb{R}}

\newcommand\NN{\mathbb{N}}

\newcommand\dx{\,\text{d}x}
\newcommand\dw{\,\text{d}w}
\newcommand\ds{\,\text{d}s}

\newcommand\dt{\,\text{d}t}
\newcommand\dy{\,\text{d}y}

\newcommand{\W}[2]
{{
		W^{#1,#2}_
		{
			\ifthenelse
			{
				\equal{#1}{2}
			}
			{
				\nu
			}
			{
			}
		}(\Omega)
}}

\newcommand{\Li}{{L^\infty(\R)}}

\newcommand{\Lp}[1]
{{
		L^
		{
			\ifthenelse
			{
				\isempty{#1}
			}
			{
				p
			}
			{
				#1
			}
		}(\R)
}}

\definecolor{lime}{HTML}{A6CE39}
\DeclareRobustCommand{\orcidicon}{%
	\begin{tikzpicture}
		\draw[lime, fill=lime] (0,0) 
		circle [radius=0.16] 
		node[white] {{\fontfamily{qag}\selectfont \tiny ID}};
		\draw[white, fill=white] (-0.0625,0.095) 
		circle [radius=0.007];
	\end{tikzpicture}
	\hspace{-2mm}
}

\definecolor{lime}{HTML}{A6CE39}
\DeclareRobustCommand{\orcidicon}{%
	\begin{tikzpicture}
		\draw[lime, fill=lime] (0,0) 
		circle [radius=0.16] 
		node[white] {{\fontfamily{qag}\selectfont \tiny ID}};
		\draw[white, fill=white] (-0.0625,0.095) 
		circle [radius=0.007];
	\end{tikzpicture}
	\hspace{-2mm}
}

\begin{document} 

\title[Pressureless Euler alignment system]{Solutions at vacuum and rarefaction waves \\ in pressureless Euler alignment system}

\author[S. Cygan]{Szymon Cygan \href{https://orcid.org/0000-0002-8601-829X}{\orcidicon}}
\address[S. Cygan]{	Instytut Matematyczny, Uniwersytet Wroc\l{}awski, pl. Grunwaldzki 2/4, \hbox{50-384} Wroc\l{}aw, Poland \\ 
\href{https://orcid.org/0000-0002-8601-829X}{orcid.org/0000-0002-8601-829X}}
\email{szymon.cygan@math.uni.wroc.pl}
\urladdr {http://www.math.uni.wroc.pl/~scygan}

\author[G. Karch]{Grzegorz Karch \href{https://orcid.org/0000-0001-9390-5578}{\orcidicon}}
\address[G. Karch]{	
Instytut Matematyczny, Uniwersytet Wroc\l{}awski, pl. Grunwaldzki 2/4, \hbox{50-384} Wroc\l{}aw, Poland \\ 
\href{https://orcid.org/0000-0001-9390-5578}{orcid.org/0000-0001-9390-5578}}
\email{grzegorz.karch@math.uni.wroc.pl}
\urladdr {http://www.math.uni.wroc.pl/~karch}

\date{\today}

\begin{abstract}
We construct global-in-time weak solutions to the pressureless Euler alignment system posed on the whole line and supplemented with 
initial conditions, where an initial density is an arbitrary, nonnegative, bounded, and integrable function (hence density at  vacuum is allowed)
and the corresponding initial velocity is determined by certain inequalities.
Moreover, our setting covers the case where solutions to the pressureless Euler alignment system are known to be non-smooth.

We also study an asymptotic behavior of constructed solutions and we show that, under a suitable rescaling, the density looks 
 like a uniform distribution on a bounded, time dependent, expanding-in-time  interval and 
 the corresponding velocity approaches a rarefaction wave ({\it i.e.}~the well-known explicit solution to the inviscid Burgers equation).
\end{abstract}

\thanks{The authors acknowledge the helpful comments on this work by Jan Peszek, Changhui Tan and Bartosz Wr\'oblewski.}
\subjclass{35Q31; 35B65; 35R11; 76N10; 82C22}

\keywords{Pressureless Euler alignment system, global-in-time weak solutions, solutions at vacuum, rarefaction waves}

\maketitle

\section{Introduction} 
\subsection*{Main results of this work} 
We study properties of global-in-time  solutions to  the following  one dimensional (pressureless)
Euler alignment system
\begin{nalign}
    \label{eq:URhoPhi}
        \begin{aligned}
        \rho_t + (\rho u )_x &= 0,\\
        u_t+ uu_x &= \int_\R \varphi(y-x) \big( u(y,t) - u(x,t)\big) \rho (y,t) \dy, 
        \end{aligned}
    \qquad x\in\R,  \quad t>0,
\end{nalign}
with the singular interaction kernel
\begin{nalign}\label{eq:varphi}
\varphi(x) = \frac{1}{\smash{|x|}^{1+\alpha}}\qquad  \text{for some}\quad  \alpha \in (0,1)
\end{nalign}
which describes the evolution of the density $\rho=\rho(x,t)$ and the corresponding velocity $u=u(x,t)$.
Notice that the right-hand side of  second equation in system \eqref{eq:URhoPhi} with singular kernel \eqref{eq:varphi} resembles the fractional Laplacian defined by the formula
\begin{nalign}
    \label{eq:frac}
    \Lambda^\alpha f(x)= \int_{\mathbb{R}} \frac{f(x)-f(y)}{|y-x|^{\alpha+1}}  \dy,
\end{nalign}  
which we often use in this work. 
We supplement this system with initial conditions 
\begin{nalign} \label{ini:0}
    \rho(\cdot,0)= \rho_0 \in \Li \cap \Lp{1},  \qquad   u(\cdot,0)= u_0 \in \Li
\end{nalign}
and under the assumption that there exists a constant $a>0$ such that 
\begin{nalign} \label{ini:G0rho0}
    0\leqslant G_0(x)\equiv (u_0)_x(x)-\Lambda^\alpha \rho_0(x) \leqslant a \rho_0(x) \quad \text{a.e.~in $x\in\R$,}
\end{nalign}
 we construct a global-in-time weak solution 
 $$u\in L^\infty( \R\times [0,\infty)), \qquad \rho \in L^\infty\big([0,\infty), \Lp{1} \cap \Li\big)$$
 to initial value problem \eqref{eq:URhoPhi}-\eqref{ini:0}
(see Theorem \ref{thm:GlobExist}, below).

Then, in Theorem \ref{thm:asymp},  we study an asymptotic behavior of constructed solutions.
Assuming, moreover,  that there are  constants $a>0$, $b>0$ such that 
\begin{nalign} \label{ini:rho0Gorho0}
    0\leqslant b \rho_0(x)\leqslant G_0(x) \leqslant a \rho_0(x) \quad \text{a.e.~in $x\in\R$}
\end{nalign}
we show
that 
 \begin{nalign}
        \lambda \rho(\lambda x, \lambda t) \xrightarrow{\lambda \to \infty} \overline\rho(x,t) = \frac{M_\rho}{M_G}
 \begin{cases}
            0, & x\leqslant 0, \\
             \frac{1}{t}, & 0<x\leqslant M_G t, \\
            0, & x> M_G t,
        \end{cases}
    \end{nalign}
and
\begin{nalign} \label{u:ini:limit}
        u(\lambda x, \lambda t) &\xrightarrow{\lambda\to\infty} \overline u(x,t) =
         \begin{cases}
            0, & x\leqslant 0, \\
             \frac{x}{t}, & 0<x \leqslant M_G t, \\
            M_G, & x > M_G t,
        \end{cases}
    \end{nalign}
in suitable (weak) typologies and 
with the numbers
$$M_\rho=\int_\R \rho_0(x)\dx \quad \text{and} \quad  M_G =\int_\R G_0(x)\dx.$$

Let us briefly summarize how the results of this work contribute to the theory on the one-dimensional pressureless Euler alignment system \eqref{eq:URhoPhi}-\eqref{eq:varphi} posed on the whole line $\mathbb{R}$.

We have obtained global-in-time weak solutions corresponding to  initial conditions where the nonnegative density $\rho_0 \in L^1(\mathbb{R}) \cap L^\infty(\mathbb{R})$ is arbitrary and $u_0$ is determined by inequalities \eqref{ini:G0rho0} which include solutions at vacuum, {\it i.e}  initial densities satisfying  $\rho_0(x)=0$ at some $x\in \R$ are allowed. 
 In this range of initial conditions, we construct weak solutions using methods analogous to those developed for the porous medium equation with a non-vocal pressure (see equation  \eqref{eq:porous}, below). Our functional setting covers the case when solutions of system \eqref{eq:URhoPhi}-\eqref{eq:varphi} are not necessarily smooth (which is known in the case when $G_0\equiv 0$ and when system \eqref{eq:URhoPhi}-\eqref{eq:varphi} reduces to the porous medium equation with a non-local pressure, see the references below).

We also study the asymptotic behavior of the constructed solutions. 
If initial conditions satisfy inequalities \eqref{ini:rho0Gorho0}, 
after suitable rescaling of a solution to system \eqref{eq:URhoPhi}-\eqref{eq:varphi}, the density $\rho(x,t)$ looks like a uniform distribution on the interval $[0, M_G t]$ and the corresponding velocity 
$u(x,t) $ approaches 
the rarefaction wave  \eqref{u:ini:limit}  which is the well-known solution of the Burgers equation $u_t+uu_x=0$.
Here, the lower bound $  0\leqslant b \rho_0(x)\leqslant G_0(x)$ in inequalities  \eqref{ini:rho0Gorho0} with $b>0$ and with a nontrivial $\rho_0$ plays a pivotal role 
because, for $G\equiv 0$, scaling limits of solutions to system \eqref{eq:URhoPhi}-\eqref{eq:varphi} is different (see Remark \ref{rho:to:Barenblatt}, below).

In this work, we consider  system \eqref{eq:URhoPhi}
with the singular kernel \eqref{eq:varphi} with $\alpha \in (0,1)$ because only  in this range, the quantity  
$$
G\equiv u_x-\Lambda^\alpha \rho \quad 
\text{written in the form} \quad  \Lambda^{-\alpha} u_x \equiv \Lambda^{-\alpha} G+\rho 
$$
allows us to obtain  suitable compactness estimates of $u(x,t)$ in both results: on the existence of solutions and on their asymptotic behavior.

A detailed statement of these results is included in the next section.

\subsection*{Cucker–Smale model and Euler alignment system}
Now, we  briefly review the motivations for studying model \eqref{eq:URhoPhi} and we cite related references. 
We refer to the recent monograph by Shvydkoy \cite{Shvydkoy} for further details and additional references.

 Model \eqref{eq:URhoPhi}  arises as the macroscopic realisation of the Cucker and Smale agent-based model 
\begin{nalign} \label{CS:system}
    x_i' = v_i, \qquad
    mv_i' = \frac{1}{N} \sum_{j=1}^N \varphi(|x_i - x_j|)(v_j - v_i)
\end{nalign}
which describes the collective motion of $N$ individuals.
This system  with a smooth kernel $\varphi$ was first introduced by Cucker and Smale \cite{CS07a,CS07b}
 who proved that  every solution aligns
exponentially and flocks strongly, namely, it has the properties:
\begin{nalign}
\lim_{t\to\infty}|v_i-\overline v|=0 \quad\text{and}\quad 
\sup_{i,j}|x_i-x_j|\leqslant D<\infty,
\end{nalign}
see the monograph \cite{Shvydkoy} as well as the papers  
\cite{AB19,MMPZ19,MT14}
 for a survey of known results.

For a large number of agents, {\it i.e.} when $N \to \infty$,
it is more efficient to consider the
mesoscopic description of the Cucker-Smale dynamics. The corresponding
kinetic formulation of system \eqref{CS:system} can be derived formally via the BBGKY hierarchy to obtain
 the following Vlasov-type model which
describes evolution of a probability distribution $f (x, v, t)$ of agents in a phase
space $(x, v)$:
\begin{equation}\label{CS;kinetic}
\partial_t f +v\cdot\nabla  f +\lambda \nabla_v\big(fF(f)\big)=0,
\end{equation}
where
\begin{equation}
    F(f)(x,v,t)=\int_{\R^{2n}} \varphi (x, y)(w - v)f (y, w, t) \dw \dy.
\end{equation}
This kinetic alignment model was first derived by Ha and Tadmor \cite{HT08} and then studied {\it e.g.}~in  
\cite{HL09,MFRT10}.

By taking $v$-moments of equation \eqref{CS;kinetic}
and introducing the monokinetic density ansatz
concentrated at the macroscopic velocity $u$:
$$
f (x, v, t) = \rho (x, t)\delta (v - u(x, t))
$$
we obtain either the system 
\begin{align*}
    \rho_t  +\nabla \cdot (\rho u) &= 0, \\
(\rho u)_t + \nabla_x \cdot  (\rho u \otimes u) &=
\int_{\R^n}
\rho (x)\rho (y)(u(y) - u(x))\varphi (x - y) \dy,
\end{align*}
or, after formally dividing by $\rho$,
the system 
\begin{align}\label{EA:1}
    \rho_t  +\nabla \cdot (\rho u) &= 0, \\
u_t + u \cdot \nabla_x u   &=
\int_{\R^n}
\rho (y)(u(y) - u(x))\varphi (x - y) \dy,\label{EA:2}
\end{align}
which is called {\it  the pressureless Euler alignment system} and which 
describes the evolution of the density profile $\rho=\rho(x,t)$  together with the  associated velocity $u(x,t)$. 
This system was derived by Ha and Tadmor \cite{HT08} (see also \cite{KV15,FK19}), and initial results regarding the existence of solutions (either local-in-time or global-in-time) with a smooth kernel $\varphi$ and smooth, initial conditions have been presented in the works \cite{HKK14,HKK15,LS22}.

Critical threshold conditions for the global regularity of solutions in the one-dimensional version of system \eqref{EA:1}--\eqref{EA:2}, posed on the entire real line $\mathbb{R}$ with a bounded positive kernel $\varphi$, first appeared in the work \cite{TT14}. The sharp criterion was later identified in the paper \cite{CCTT16}, accompanied by the proof that global-in-time regular solutions exist if and only if the initial data are ``sub-critical" in the following sense:
\begin{equation} \label{tres}
(u_0)_x(x)+\varphi*\rho_0(x) \geq 0 \quad \text{for all} \quad x\in \mathbb{R}.
\end{equation}
On the other hand, as seen in \cite{ST17}, the one-dimensional model posed on the one-dimensional torus $\mathbb{T}$ with the singular kernel $\varphi(x)=|x|^{-1-\alpha}$ for $1\leqslant \alpha < 2$, and subject to initial data $(u_0, \rho_0) \in  
H^3(\mathbb{T}) \times H^{2+\alpha}(\mathbb{T})$, admits a global solution within the same class. Furthermore, still considering the one-dimensional model on the torus but with $\alpha \in (0,1)$, it was proven in \cite{DKRL18} that the system, supplemented with smooth initial data $(\rho_0,u_0)$ such that $\rho_0(x) > 0$ for all $x \in \mathbb{T}$, possesses a unique global smooth solution.
Other results regarding the existence and regularity of solutions to the one-dimensional model with $\alpha\in(0,2)$, on the torus and with no-vacuum initial data  are contained in \cite{ST17b,ST18,L19}. It is noteworthy that the singularity of the kernel $\varphi(x)=|x|^{-1-\alpha}$
removes the requirement \eqref{tres}. The regularity result for $\alpha \in (0,1)$, first obtained in \cite{DKRL18}, is even more surprising, because for the analogous model - the fractal Burgers equation - solutions develop shocks in finite time, as shown in \cite{ADV07}.
 
Finally, we recall recent results concerning the existence and regularity of solutions in the multidimensional case. In the work \cite{DMPW19}, the system \eqref{EA:1}--\eqref{EA:2} with $\alpha \in (1,2)$ is examined on the entire space $\mathbb{R}^n$ with $n\geqslant 1$, and global-in-time regular solutions have been constructed under a smallness assumption imposed on $(u_0, \rho_0-1)$. Then, the authors of \cite{RS20} (see also \cite{LRS22,L23} and the references therein) consider this system with $\alpha \in (0,2)$ on the $n$-dimensional torus, obtaining local-in-time regular solutions corresponding to sufficiently regular initial conditions with no vacuum  
({\it i.e.}~for $\rho_0(x)>0$ for all $x$). Global-in-time weak solutions to the forced Euler-alignment system on the torus and without vacuum (this time, under the assumption $\rho_0^{-1}\in L^\infty(\mathbb{T})$) have been constructed in \cite{L19}.

\subsection*{Porous medium equation with the  nonlocal pressure and Barenblatt self-similar profiles}
In this work, we study
the Euler alignment  system \eqref{eq:URhoPhi}-\eqref{eq:varphi} written in the   equivalent form (see comments below equations \eqref{eq:URhoG})
\begin{nalign}
        \rho_t +(\rho u)_x  = 0, \quad
        G_t + (Gu)_x  = 0, \quad
        u_x - \Lambda^\alpha\rho = G,
\end{nalign}
which for $G\equiv 0$
reduces to
 the so-called one dimensional
porous medium equation with the  nonlocal pressure 
    \begin{nalign} \label{eq:porous}
        \rho_t + (u\rho)_x = 0 \quad\text{with} \quad 
         u_x =  \Lambda^\alpha \rho \quad\text{for} \quad  x\in \R,\; t>0.
    \end{nalign}
This equation, supplemented with an initial condition
\begin{equation} \label{npm}
    \rho(\cdot,0)=\rho_0 \in L^1(\R)\cap L^\infty(\R)
\end{equation}
and motivated by the theory of dislocations, was considered by  Biler, Karch, and Monneau \cite{BKM10}, who constructed a unique global-in-time viscosity solution $\rho =\rho(x,t)$. Then, the extension of model \eqref{npm} to higher dimensions was introduced 
 by Caffarelli and V\'azquez starting in  \cite{CV11a,CV11b} (followed by several other  papers, see {\it e.g.} 
\cite{CSV13,CV15,BIK15,STV16,STV19}
and references therein), where  the existence of weak solutions,
the finite speed of propagation, local H\"older regularity, and asymptotic behaviour
were established.

\begin{figure}[t]
  \centering
    \includegraphics[width =0.30\textwidth]{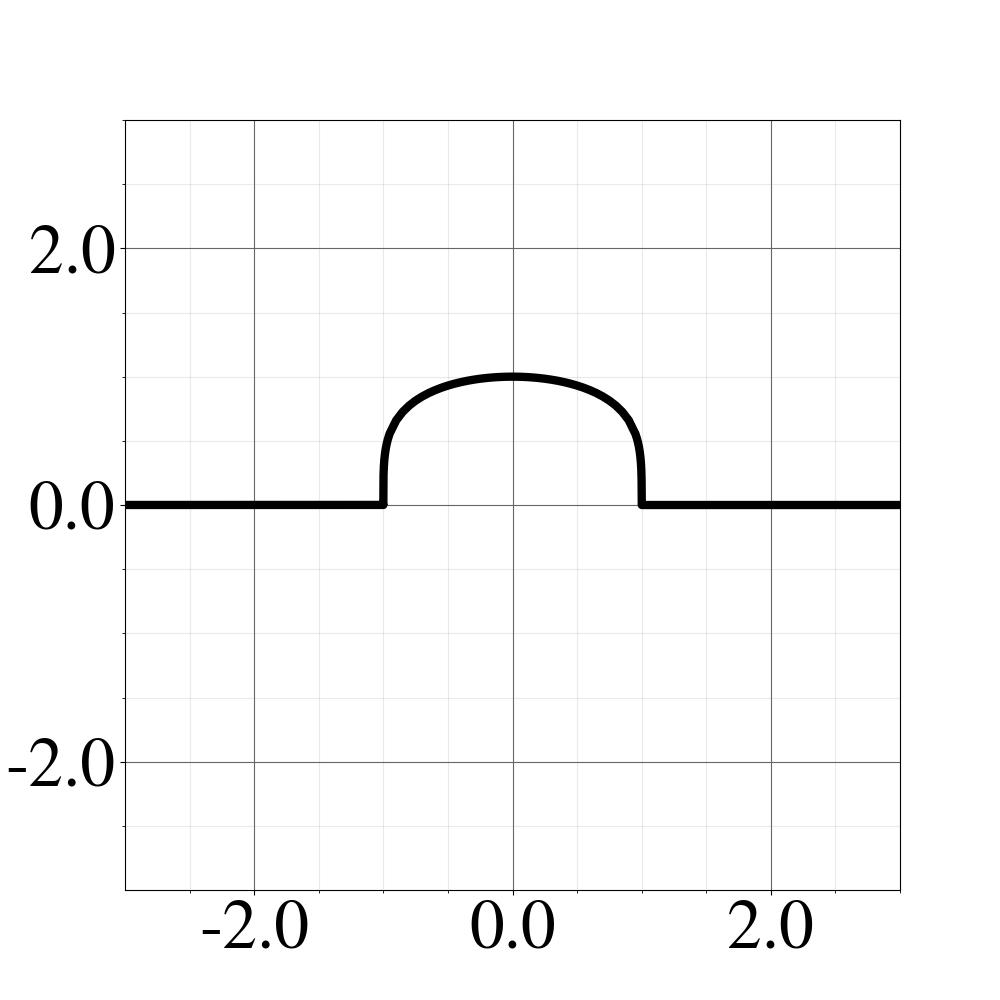} 
    \includegraphics[width =0.30\textwidth]{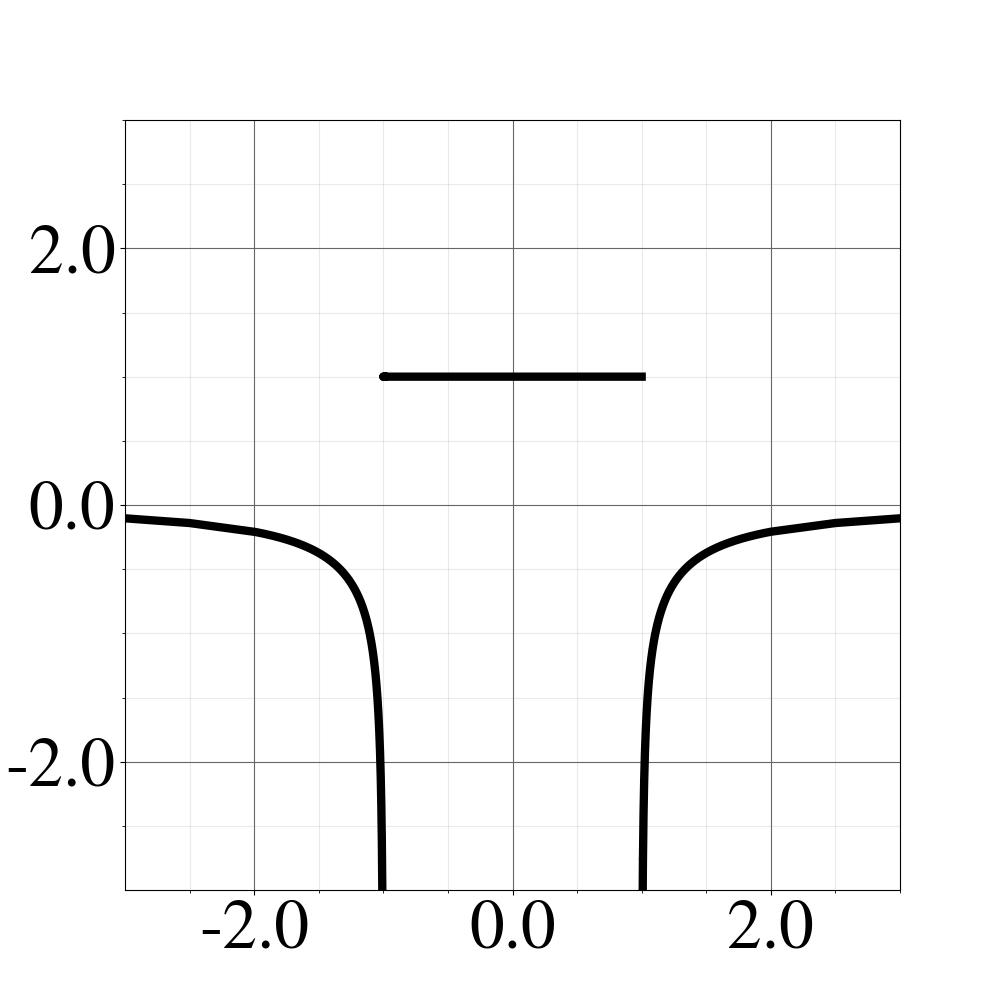}        
     \includegraphics[width = 0.30\textwidth]{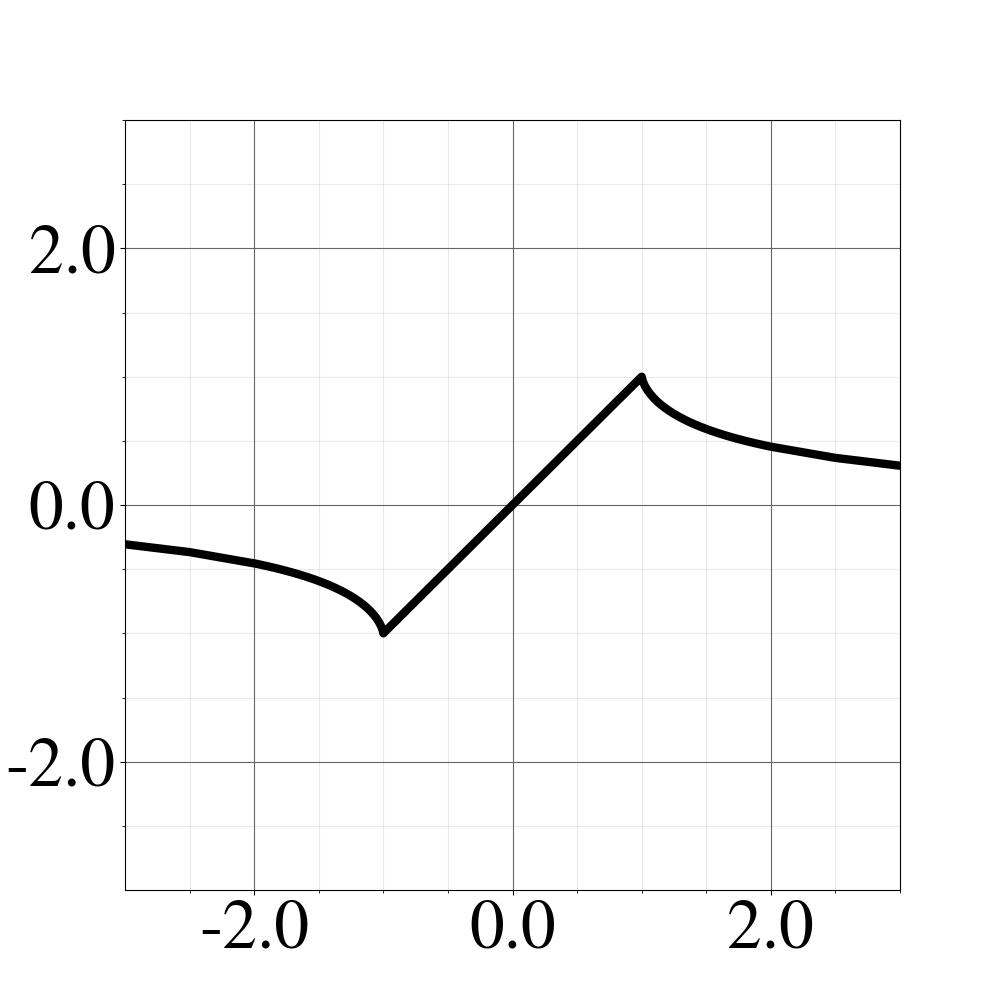}
       \caption{From the left to the right: the Getoor function  $\Phi_\alpha (x)$ given by formula \eqref{profiles} with $\alpha=1/2$,
        the corresponding $\Lambda^\alpha \Phi_\alpha(x)$ (see \eqref{LaPa}, and 
     $U(x)=\int_{-\infty}^x \Lambda^\alpha \Phi_\alpha (y)\; dy$.
       }
    \label{fig:rho0}
\end{figure}

In particular, the 
 work  \cite{BKM10}
contains a construction of 
the following  explicit self-similar solution of equation \eqref{eq:porous} 
    \begin{nalign}\label{rho:self}
        \rho(x,t) = t^{-\frac{1}{\alpha+1}}\Phi_\alpha\left( {x}{t^{-\frac{1}{\alpha+1}}}\right), 
    \end{nalign}
with the profile 
\begin{equation} \label{profiles}
 \Phi_\alpha(x) = K(\alpha,1)  \left(1- |x|^2\right)_+^\frac{\alpha}{2}, 
\end{equation}
and the constant 
\begin{equation}\label{Ka1}
K(\alpha,1) = \Gamma\left(\frac{1}{2}\right) \left[2^\alpha \Gamma\left(1 + \frac{\alpha}{2} \right) \Gamma\left( \frac{1+\alpha}{2}\right) \right]^{-1},
\end{equation} 
where $\Gamma$ denotes the usual $\Gamma$-function 
(see also \cite{BIK15} for analogous  explicit self-similar solutions of the higher dimensional counterpart of  equation \eqref{eq:porous}).
This explicit solution  generalizes  the well-known  Barenblatt self-similar solution of the classical porous medium equation. 
On the other hand,  $\Phi_\alpha (x)$ is called the Getoor function and it has a probabilistic interpretation:
it express the 
expected time of the first passage to the exterior of the unit ball of the symmetric $\alpha$-stable process starting at $x$.


    It was proved by Getoor \cite[Thm 5.2]{getoor1961first} that
    \begin{nalign} \label{LaPa}
        \Lambda^\alpha \Phi_\alpha (x)= 
        \begin{cases}
            1, & \text{for} \quad  |x| <1, \\
            H(x), & \text{for} \quad |x|>1,
        \end{cases}
    \end{nalign}
    where 
    $$H(x) =  \left[ \Gamma\left( -\frac{\alpha}{2}\right) \Gamma \left( \frac{3+\alpha}{2}  \right) \right]^{-1} \Gamma \left( \frac{1}{2}\right) \cdot |x|^{-1-\alpha} \cdot F\left( \frac{2+\alpha }{2}, \frac{1+\alpha}{2}, \frac{1+\alpha}{2}, |x|^{-2}  \right)$$ and $F(a,b,c,x)$ is the hypergeometric function. 
     Thus, applying  the relation  $u_x =\Lambda^\alpha \rho$ to  the self-similar solution  
     \eqref{rho:self}  we obtain the velocity in the explict self-similar form 
       \begin{nalign}\label{uU}
        u(x,t)= U\left( {x}{t^{-\frac{1}{\alpha+1}}}\right) \quad\text{with}\quad  U(x) =\int_{-\infty}^x \Lambda^\alpha  \Phi_\alpha(y)\; \dy.
    \end{nalign}
    Here, we emphasize that,  
 by  properties of the function $H=H(x)$ discussed in 
    \cite[Thm.~5.2]{getoor1961first},
     we obtain 
    \begin{itemize}
        \item $U_x(x) <0$  for $|x|  \geqslant 1$;
        \item $U_x(x) \sim \frac{C}{|x|^{1+\alpha}}$ when $|x| \to \infty$;
        \item $U_x (x)\sim C\left( 1 - |x|^2 \right)^{-\frac{\alpha}{2}}$ when   $|x|  \geqslant 1$ and $|x| \to 1$.
    \end{itemize}
%
The explicit self-similar solution  \eqref{rho:self} and \eqref{uU}, with the profiles $\Phi_\alpha$ and $U$ shown in Fig.~\ref{fig:rho0}, describes an asymptotic behavior of other solutions which we discuss below in Remark \ref{rho:to:Barenblatt}.

\begin{remark}
Notice that if $\rho_0\in C^1(\R)$ is nonnegative and compactly supported, then the corresponding initial velocity
has similar properties:
\begin{equation}
(u_0)_x (x)=\Lambda^\alpha \rho_0(x)  \leqslant 0 \quad\text{for all} \quad x\in \R\setminus {\rm supp}\, \rho_0
\end{equation}
and
$$
|\partial_x u_0(x)| \leqslant \frac{C}{|x|^{1+\alpha}} \quad\text{for all} \quad x\in \R,
$$
which is a direct consequence of the definition of the fractional Laplacian  \eqref{eq:frac}. 
\end{remark}



\subsection*{Notation}
For $p \in  [1, \infty]$, we denote the norm in the usual Lebesgue space $L^p(\R)$
by $\|\cdot\|_p$. The standard Sobolev space of order $k$ is written as $W^{k,p}(\R)$.
For other Banach spaces $X$ appearing throughout this work, the corresponding norm is denoted by $\|\cdot\|_X$. 
The symbol $\mathcal{S}(\R)$ represents the Schwartz space of smooth, rapidly decreasing functions, while $\widehat{\cdot}$ and $\widecheck{\cdot}$ denote the Fourier transform and its inverse. Additionally, for a function $v\in \Lp{1}$, we use the shorthand
\begin{nalign}
    \partial_x^{-1} v(x) = \int_{-\infty}^x v(y) \dy.
\end{nalign}
Constants may vary from line to line and will be denoted by $C$.

\section{Results and comments}

\subsection*{Statement of the problem}
We begin by writing the Euler alignment system with the strongly singular interaction kernel 
\begin{nalign}
    \label{eq:URhoConv}
    \begin{aligned}
        \rho_t + (\rho u )_x &= 0,  \\
        u_t+ uu_x &= \int_\R \frac{1}{|y-x|^{1+\alpha}} \big( u(y,t) - u(x,t)\big) \rho (y,t) \dy, 
    \end{aligned}
    \qquad x\in \R, \quad t>0,
\end{nalign}
 using the relation
\begin{nalign}\label{rel:0}
    \big( u(y) - u(x) \big) \rho(y) =  u(x) \big(\rho(x) - \rho(y)\big) -\big(u(x)\rho(x) - u(y) \rho(y) \big)
\end{nalign}
and the definition of the fractional Laplacian \eqref{eq:frac}  
as the system 
\begin{nalign}
    \label{eq:URhoFractional}
    \begin{aligned}
        \rho_t + (\rho u )_x &= 0,  \\
        u_t+ u(u_x - \Lambda^\alpha\rho) +\Lambda^\alpha(\rho u)&=0,
    \end{aligned}
    \qquad x\in \R, \quad t>0,
\end{nalign}
which we supplement with initial conditions 
\begin{nalign}
    \label{eq:iniConv}
    \rho(x,0) = \rho_0(x) \quad \text{and} \quad u(x,0) = u_0(x).
\end{nalign}

It is useful to study 
 system \eqref{eq:URhoFractional}  expressed (by direct calculations)  in the following formally equivalent form 
\begin{nalign}
    \label{eq:URhoG}
    \begin{aligned}
        \rho_t +(\rho u)_x & = 0, \\
        G_t + (Gu)_x & = 0, \\
        u_x - \Lambda^\alpha\rho& = G,
    \end{aligned}
    \qquad x\in \R, \quad t>0,
\end{nalign}
together with initial conditions
\begin{nalign}
    \label{eq:ini3eq}
    \rho(x,0) = \rho_0(x) \quad \text{and} \quad G(x,0) = G_0(x).
\end{nalign}

\subsection*{Existence of solutions}
Our first goal is to construct  
functions  
$$
(\rho, u,G)= \big(\rho(x,t),u(x,t),G(x,t)\big) \quad \text{for all} \quad  x\in\R, \quad  t\geqslant 0,
$$
 which satisfy both systems \eqref{eq:URhoFractional} and \eqref{eq:URhoG} in a weak sense, with $\alpha \in (0,1)$.
 In fact, we can show that a weak solution of problem \eqref{eq:URhoG}--\eqref{eq:ini3eq}
 solves problem
\eqref{eq:URhoFractional}--\eqref{eq:iniConv}
in the same way as in the proof of Proposition \ref{prop:regular}, below.
In the case when
the interaction kernel $\varphi(x)$ is an integrable  and bounded function, 
the authors of \cite{CCTT16, ST17} noticed that, these two systems are indeed equivalent which played  an important role in their analysis. 

We begin by definitions of weak solutions.
 
\begin{definition} 
    \label{def:WeakSol2eq}
    A couple $(u, \rho)$ is a weak solutions to problem \eqref{eq:URhoFractional}--\eqref{eq:iniConv} with
    initial conditions
    $\rho_0 \in \Li \cap \Lp{1}$  and   $u_0 \in \Li$
    if
    \begin{itemize}
        \item $u\in L^\infty(\R\times [0,\infty))$ and $\rho \in L^q\big([0,T], \Lp{} \big)$  for all $p,q \in (1,\infty)$, $T>0$;
        \item $u \in C\big( [0,T], L^q([-R,R])\big)$ for each $T>0$, $R>0$, and $q\in\big(1/(1-\alpha),\infty\big)$;
        \item $u_x - \Lambda^\alpha \rho \in L^q\big([0,T], \Lp{} \big)$  for all $p,q \in (1,\infty)$ and each $T>0$;
        \item the following equations hold true
        \begin{nalign}
            \int_0^\infty \int_\R \rho \varphi_t \dx \dt + \int_\R \rho_0 \varphi(\cdot,0) \dx &+ \int_0^\infty \int_\R \rho u \varphi_x \dx \dt = 0, \\
            \int_0^\infty \int_\R u \psi_t \dx \dt + \int_\R u_0 \psi(\cdot,0) \dx &
            - \int_0^\infty \int_\R  u(u_x - \Lambda^\alpha \rho)\psi \dx \dt   \\ 
            &- \int_0^\infty \int_\R  u\rho  \Lambda^\alpha\psi \dx \dt = 0,
        \end{nalign}        
    \end{itemize}
    for all  $\varphi, \psi \in C_c^\infty\big(\R \times [0,\infty)\big)$.
\end{definition}

\begin{definition}
    \label{def:WeakSol3eq}
    A triple $(u,\rho, G)$ is a weak solution to problem \eqref{eq:URhoG}--\eqref{eq:ini3eq} with initial conditions $\rho_0, G_0 \in \Li \cap \Lp{1}$  if 
    \begin{itemize}    
        \item $u\in L^\infty(\R\times [0,\infty))$ and $\rho, G \in L^q\big([0,T], \Lp{} \big)$  for all $p,q \in (1,\infty)$, $T>0$;
        \item  $u \in C\big( [0,T], L^q([-R,R])\big)$ for each $T>0$, $R>0$ and $q\in\big(1/(1-\alpha),\infty\big)$;
        \item the following equations hold true
    \end{itemize}
    \begin{nalign} \label{rho:G:u:weak}
        \int_0^\infty \int_\R \rho \varphi_t \dx \dt + \int_\R \rho_0 \varphi(\cdot,0) \dx + \int_0^\infty \int_\R \rho u \varphi_x \dx \dt &= 0, \\
        \int_0^\infty \int_\R G \psi_t \dx \dt + \int_\R G_0 \psi(\cdot,0) \dx + \int_0^\infty \int_\R G u \psi_x \dx \dt &= 0, \\
         \int_0^\infty \int_\R u \zeta_x \dx \dt + \int_0^\infty \int_\R \rho \Lambda^\alpha \zeta \dx \dt + \int_0^\infty \int_\R G \zeta \dx \dt&=0,
    \end{nalign}
    for all $\varphi,\psi,\zeta \in C_c^\infty\big(\R \times [0,\infty)\big)$.
\end{definition}

We denote the masses of initial conditions by

\begin{nalign} \label{mass}
    M_\rho \equiv  \int_\R{\rho_0} (x) \dx \quad \text{and} \quad M_G \equiv \int_\R{G_0(x)}\dx.
\end{nalign}

\begin{theorem}[Existence of solutions]
    \label{thm:GlobExist}
    Let $\alpha \in (0,1]$.
    Assume that $u_0 \in \Li$ and $\rho_0 \in \Li \cap \Lp{1}$ satisfy 
    $G_0 \equiv  \partial_x u_0 - \Lambda^\alpha \rho_0 \in \Li \cap \Lp{1}$.
    Suppose  there exist a constant $a>0$ such that
    \begin{nalign}\label{ini:exist}
        0 \leqslant G_0(x) \leq a \rho_0(x) \quad \text{a.e.~in} \quad x\in \R.
    \end{nalign}
    There exist a global-in-time weak solution $(\rho,u)$ of problem \eqref{eq:URhoFractional}--\eqref{eq:iniConv} (cf. Def.~\ref{def:WeakSol2eq}) which for $G=u_x-\Lambda^\alpha \rho$,  
    is also a weak solution of problem \eqref{eq:URhoG}--\eqref{eq:ini3eq} (cf. Def.~\ref{def:WeakSol3eq}). This solution satisfies  the following properties for all $t>0$:
        \begin{itemize}
        \item the estimate 
        \begin{nalign} \label{u:inf:0}
             \|u(\cdot, t)\|_\infty &\leqslant \|u_0\|_\infty;
        \end{nalign}
        \item the positivity and comparison principles \label{Py}
        \begin{nalign} \label{Garho:0}
            0\leqslant G(x,t) \leqslant a \rho(x,t) \quad   \text{a.e. in} \quad x\in\R;
        \end{nalign}
        \item the $L^p$-estimates 
        \begin{nalign}\label{eq:Lp:est}
              \frac{1}{a}\|G(\cdot, t)\|_p &\leqslant\|\rho(\cdot, t)\|_p \leqslant \|\rho_0\|_p \quad \text{for each} \quad p\in[1,\infty]
        \end{nalign}
        and
        \begin{nalign}
            \label{eq:DecEstim}
            \frac{1}{a}\|G(\cdot, t)\|_p \leqslant\|\rho(\cdot, t)\|_p \leqslant C M_\rho^\frac{2+p\alpha}{2p + 2\alpha} {t^{\left(-1 + \frac{1}{p}\right)\frac{1}{2+\alpha}}} \quad \text{for each} \quad p\in(2,\infty).
        \end{nalign}
    \end{itemize}
\end{theorem}

\begin{remark} \label{rem:u:const:0}
    In the proof of Theorem \ref{thm:GlobExist}, we construct solutions of system \eqref{eq:URhoG} where the last equation is replaced by its integrated version 
    $u = \partial^{-1}_x G + \partial^{-1}_x \Lambda^\alpha \rho$. Notice that, if~$(\rho,u)$ is a solution of the Euler alignment system \eqref{eq:URhoFractional} then so is the pair $(\rho, u+C)$ for each constant $C\in \R$
    (it suffices to  change the variables $(x,t)\mapsto (x-Ct,t)$). Thus, without loss of generality we may ``normalize'' the solution $u(x,t)$ by assuming that $\lim\limits_{x\to -\infty} u(x,t)=0$ (provided  this limit exists). 
    This normalization does not play any role neither in the proof of Theorem \ref{thm:GlobExist} (Existence of solutions) nor in Theorem \ref{thm:DecayEstiamtes} (Decay estimates). 
    However, we impose it explicitly by choosing $u_0=\partial^{-1}_x G_0+\partial^{-1}_x \Lambda^\alpha\rho_0$ in 
    Theorem \ref{thm:asymp} in order to study  the asymptotic behavior of solutions.
\end{remark}

\begin{remark} 
    We emphasize that  the pair $\big(\rho(x,t), u(x,t)\big)$, given by formulas \eqref{rho:self} and \eqref{uU}, with the profiles $\Phi_\alpha(x)$ and $U(x)$ as depicted in Fig.~\ref{fig:rho0}, represents an explicit weak {\it non-smooth}  solution of the Euler alignment system \eqref{eq:URhoConv} which falls in the scope of Theorem \ref{thm:GlobExist}. 
    Notice that, according to the papers \cite{T19,AC21}, solutions of equation \eqref{eq:porous} (thus of system \eqref{eq:URhoPhi}) may develop singularities for $\alpha \in (0,2)$, on the real line or in the periodic case, even with just a point of vacuum.
    On the other hand, as shown {\it e.g.} in~\cite{ST17,DKRL18,ST17b,ST18,L19}, solutions remain smooth for all time in the case where the Euler alignment system \eqref{eq:URhoPhi} with $\alpha \in (0,2)$ is considered on the torus and is supplemented with smooth initial data such that $\rho_0>0$. 
\end{remark}

\subsection*{Asymptotic behavior of solutions and rarefaction waves}

Methods used in the proof of Theorem \ref{thm:GlobExist} will be now applied to study 
an asymptotic  behavior of solutions. First, however,
we improve decay estimates \eqref{eq:DecEstim} under  an additional assumption on initial conditions. Below, we discuss a  class of initial 
data satisfying such assumptions.

\begin{theorem}[Decay estimates]
    \label{thm:DecayEstiamtes}
    Let $(\rho, u,  G)$  be the solution to problems  \eqref{eq:URhoFractional}--\eqref{eq:iniConv} and  \eqref{eq:URhoG}-\eqref{eq:ini3eq} constructed in Theorem \ref{thm:GlobExist} and corresponding to the initial conditions $(\rho_0, G_0)$ satisfying the following inequalities 
    for some constants  $a>0$ and $b>0${\rm :}
    \begin{nalign} 
        \label{eq:InitConComp}
        0\leqslant b \rho_0(x) \leqslant G_0(x) \leqslant a \rho_0(x) \quad \text{a.e. in $x\in \R$}.
    \end{nalign} 
  Then the following estimates hold true for all $t>0${\rm :}
        \begin{nalign}
            b \rho(x,t) \leqslant G(x,t) \leqslant a \rho(x,t)\quad \text{a.e. in $x\in \R$} 
        \end{nalign}
            and 
        \begin{nalign} 
            \label{eq:RhoGDecayEstim}
            \|\rho(\cdot, t)\|_p \leqslant C(p) M_\rho^\frac{1}{p} {t^{-1+\frac{1}{p}}}   \quad \text{and} \quad \|G(\cdot, t)\|_p \leqslant C(p) M_G^\frac{1}{p} {t^{-1+\frac{1}{p}}},
        \end{nalign}
          for    each  $p\in[1,\infty]$,
        where  constants $C = C(p)>0$ are independent of $t$ and the masses $M_\rho$ and $M_G$ are given by formulas \eqref{mass}.
\end{theorem}

The $L^p$-decay estimates \eqref{eq:RhoGDecayEstim} allow us to to study the following scaling limits of solutions to systems
\eqref{eq:URhoFractional}-\eqref{eq:iniConv} and \eqref{eq:URhoG}-\eqref{eq:ini3eq}
\begin{align}
    \rho^\lambda(x,t) \equiv \lambda \rho(\lambda x,\lambda t), \quad G^\lambda(x,t) \equiv  \lambda G(\lambda x,\lambda t), \quad u^\lambda(x,t) \equiv  u(\lambda x,\lambda t).
\end{align}

\begin{theorem}[Asymptotic behavior] 
\label{thm:asymp}
    Let $(u,\rho, G)$ be a weak solution to systems \eqref{eq:URhoFractional}-\eqref{eq:iniConv} and \eqref{eq:URhoG}-\eqref{eq:ini3eq} constructed in Theorem \ref{thm:GlobExist} and satisfying the $L^p$-decay estimates \eqref{eq:RhoGDecayEstim}. 
    Assume, moreover, that $u_0=\partial^{-1}_x G_0+\partial^{-1}_x \Lambda^\alpha\rho_0$.
    For each $r,p\in(1,\infty)$ and $q\in\big( 1/\alpha, \infty\big)$
        \begin{nalign}
        \label{eq:WeakLambdaConvergence}
        \begin{matrix}
            \rho^{\lambda} \to \overline{\rho} \\
             G^{\lambda} \to \overline{G}
        \end{matrix}& & & \text{weakly in} \; \; L^r\big( [t_1, t_2], \Lp{} \big),  \\
        u^{\lambda} \to \overline{u}& & & \text{strongly in} \; \; C\big( [t_1, t_2], L^q([-R,R]) \big),
    \end{nalign} 
for all $0<t_1<t_2$ and $R>0$,
    where $\big(\overline{u}, \overline{\rho}, \overline{G}\big)$ is a weak solution to the system
    \begin{nalign}
        \overline{\rho}_t + (\overline{u}\overline{\rho} )_x &= 0, \\
        \overline{G}_t + (\overline{u}\overline{G} )_x &= 0, \\
        \overline{u}_x=&\overline G, 
    \end{nalign}
    supplemented with the the initial conditions ${\rm(}$involving the Dirac measure $\delta_0$ and the masses of initial conditions \eqref{mass}{\rm )}
    \begin{nalign}
          \overline{\rho}(x,0) = M_\rho \delta_0, \qquad \overline{G}(x,0) = M_G \delta_0,
    \end{nalign}
    and given by the explicit formulas
    \begin{nalign}\label{orho}
       \overline{\rho}(x,t) =  \frac{M_\rho}{M_G}
        \begin{cases}
            0, & x\leqslant 0, \\
             \frac{1}{t}, & 0<x\leqslant M_G t, \\
            0, & x> M_G t,
        \end{cases} \qquad
         \overline{G}(x,t) = 
        \begin{cases}
            0, & x\leqslant 0, \\
            \frac{1}{t}, & 0<x\leqslant M_G t, \\
            0, & x> M_G t
        \end{cases} 
    \end{nalign}
and
        \begin{nalign}\label{orho:u}
\overline{u}(x,t) = 
        \begin{cases}
            0, & x\leqslant 0, \\
            \frac{x}{t}, & 0<x \leqslant M_G t, \\
            M_G, & x > M_G t.
        \end{cases}
        \end{nalign}
        
\end{theorem}

\begin{remark}
    The scaling limit  of $u^\lambda=u(\lambda x,\lambda t)$ given by the rarefaction wave $\overline{u}$ in \eqref{orho:u} resembles  the results from the paper \cite{KMX08}
    on the convergence
    toward rarefaction waves 
    of solutions to the fractal Burgers equation 
    $u_t+\Lambda^\alpha u +uu_x=0$, supplemented with step-like initial data such that 
     $\lim\limits_{x\to -\infty} u_0(x) < \lim\limits_{x\to +\infty} u_0(x).$
\end{remark}

\begin{remark}\label{rho:to:Barenblatt}
Let us emphasize that  scaling limits of solutions 
 to problems  \eqref{eq:URhoFractional}--\eqref{eq:iniConv} and  \eqref{eq:URhoG}-\eqref{eq:ini3eq} constructed in Theorem \ref{thm:GlobExist} differ depending 
 on whether $G\equiv 0$ or not. 
 Due to assumption \eqref{eq:InitConComp},
Theorem \ref{thm:asymp} covers the case of nonzero $G$. On the other hand, for $G\equiv 0$, the considered system reduces to the one dimensional
porous medium equation with the  nonlocal pressure \eqref{eq:porous} where a scaling limit was already considered in \cite[Thm.~2.5]{BKM10}.
It is proved in that work, that supplementing the porous medium equation  \eqref{eq:porous} with an initial condition such that  $\int_\R \rho_0(x)\dx=1$, the corresponding solution $\rho = \rho(x,t)$ satisfies
\begin{equation*}
     \lambda \rho(\lambda x, \lambda^{1+\alpha} t) \xrightarrow{\lambda \to \infty} t^{-\frac{1}{\alpha+1}}\Phi_\alpha\left( {x}{t^{-\frac{1}{\alpha+1}}}\right)
\end{equation*}
in a suitable topology.
\end{remark} 

\subsection*{Comments on considered initial conditions}

\begin{figure}[t]
  \centering
        \includegraphics[width =0.45\textwidth]{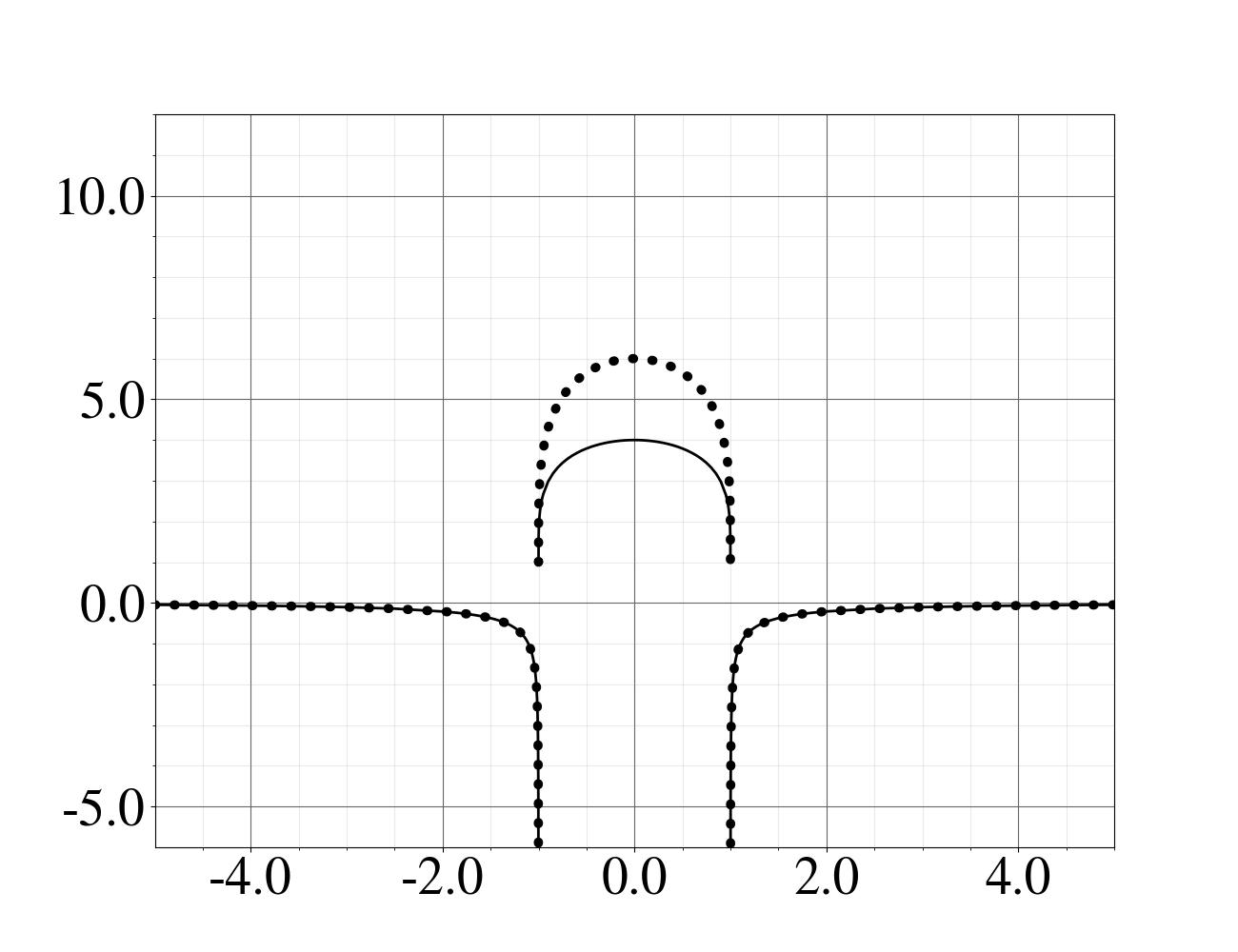}    
     \includegraphics[width = 0.45\textwidth]{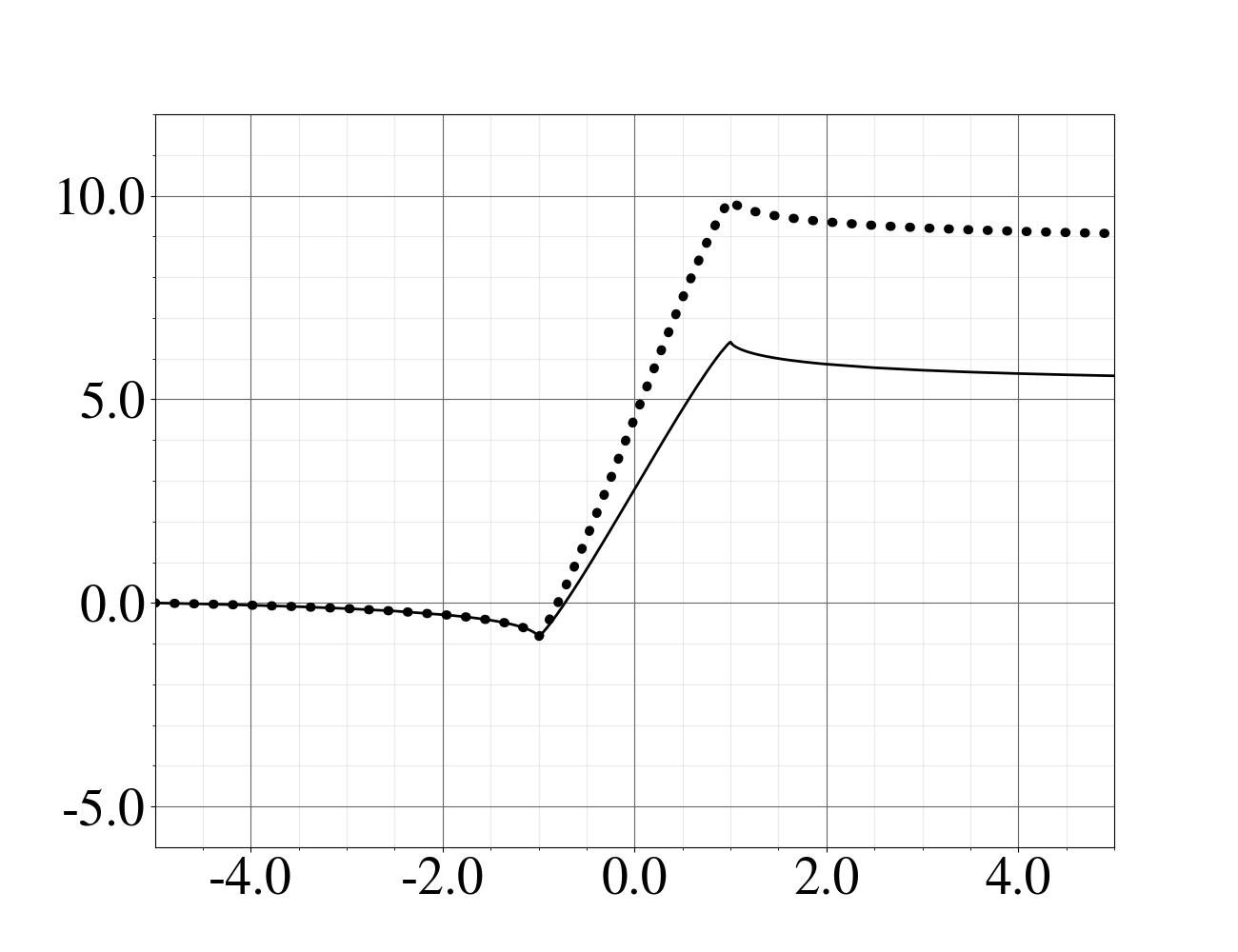}
       \caption{In order to illustrate inequalities \eqref{eq:UBounds} with $\rho_0=\Phi_\alpha$, we plot 
        $a\Phi_\alpha +\Lambda^\alpha \Phi_\alpha$   for $a=3$ (solid line) and $a = 5$ (dotted line), on the left-hand side. 
       The figure on the right-hand side shows the corresponding   $u_0(x)=\int_{-\infty}^x \big(a\Phi_\alpha (y)+\Lambda^\alpha\Phi_\alpha (y)\big)\;dy$.
       }
    \label{fig:rho}
\end{figure}

Here, we discuss assumptions \eqref{ini:exist} and \eqref{eq:InitConComp} imposed on initial conditions $\rho_0, u_0, G_0$.
Recalling that $$G_0=(u_0)_x - \Lambda^\alpha\rho_0$$
we write assumption \eqref{eq:InitConComp}  in the form
    \begin{nalign}
        \label{eq:UBounds}
        b \rho_0(x) + \Lambda^\alpha \rho_0(x)\leqslant (u_0)_x(x)\leqslant a\rho_0(x) + \Lambda^\alpha \rho_0(x)
        \qquad \text{for all} \quad x\in\R.
    \end{nalign}

First, we choose  $\rho_0$ in the form of  the Getoor function $\rho_0=\Phi_\alpha$ given by formulas 
\eqref{profiles}-\eqref{Ka1}.
 Since the Getoor function $\Phi_\alpha$ is supported  in $[-1,1]$, 
     using the properties of  $\Lambda^\alpha \Phi_\alpha (x)$ given by expression \eqref{LaPa}, we obtain that
     assumption \eqref{eq:UBounds} takes the form of two inequalities: 
     \begin{equation} \label{in:as:1}
         b \Phi_\alpha(x) +1\leqslant   \partial_x u_0(x) \leqslant a \Phi_\alpha(x) +1 \qquad \text{for} \quad   |x|\leqslant 1
     \end{equation}
     and 
  \begin{equation}\label{in:as:2}
        \partial_x u_0(x) = \Lambda^\alpha \Phi_\alpha(x) =H(x) \qquad \text{for} \quad   |x|\geqslant 1.
  \end{equation}
Fig.~\ref{fig:rho} illustrates functions satisfying these two conditions. 

  Similarly, if $\rho_0 \in C^1(\R)$ is compactly supported and positive, by the definition of the fractional Laplacian \eqref{eq:frac}, assumption  \eqref{eq:UBounds}
implies
  $$
 (u_0)_x(x)=\Lambda^\alpha\rho_0 (x) \leqslant 0 \quad\text{for all} \quad x\in \R\setminus {\rm supp}\, \rho_0.
$$
Moreover, since $\int_\R \Lambda^\alpha \rho_0(x) \dx =0$,
under the assumptions
$$
u_0=\partial^{-1}_x G_0+\partial^{-1}_x \Lambda^\alpha \rho_0 \quad\text{and} \quad 0\leq G_0\in L^1(\R),
$$
 we obtain 
$$
0=\lim_{x\to -\infty} u_0(x) < \lim_{x\to +\infty} u_0(x) = \int_\R G_0(x)\dx.
$$



\section{Auxiliary results involving non-local operators}
In this section, we will collect results used in the proofs of main theorems.

\begin{lemma}
    \label{lem:FracEstim}
    Let $\alpha\in (0,1)$. There exists a constant $C>0$ such that the inequality
    \begin{equation}
        \|\Lambda^{\alpha} f\|_1\leq C\left(\|f\|_1+\|\partial_x^2 f\|_1\right)
    \end{equation}
    holds true for all $f\in W^{2,1}(\R)$.
\end{lemma}
\begin{proof}
    It suffices to prove the inequality
    \begin{equation}
     \|\Lambda^{\alpha} f\|_1\leq C\|f - \partial_x^2 f\|_1
    \end{equation}
    which is equivalent to the estimate
    \begin{equation}
         \|\Lambda^{\alpha}(I-\partial^2_x)^{-1} g\|_1\leqslant C\|g\|_1
    \end{equation}
    for each $g\in L^1(\R)$. The  operator $\Lambda^{\alpha}(I-\partial^2_x)^{-1}$ is a pseudo-differential operator 
    and, because $\alpha\in (0,1)$, its 
    symbol $m(\xi)$
    satisfies:
    \begin{nalign}
        m(\xi) = \frac{|\xi|^\alpha}{1+|\xi|^2} \in L^1(\R)\cap L^\infty(\R)
    \end{nalign} 
and $\partial_\xi m  \in \Lp{}$ for each $p\in \big[1, \, 1/(1-\alpha)\big)$. We show below, that $\widehat{m}\in \Lp{1}$  (notice that $\widehat m=\widecheck m$)
    which will complete the proof because, by the Young inequality for convolution, we have
    \begin{nalign}
        \|\Lambda^{\alpha}(I-\partial^2_x)^{-1} g\|_1 = \No{(m\widehat{g})\widecheck{\,}} = \No{\widecheck{m} * g} \leqslant \No{\widecheck{m}} \No{g}.
    \end{nalign}
    In order to show that $\No{\widecheck{m}}$ is finite, we  decompose
    \begin{nalign}
        \int_\R |\widecheck{m}(x)| \dx = \int_{|x|\leqslant 1} |\widecheck{m}(x)| \dx + \int_{|x|>1} |\widecheck{m}(x)| \dx.
    \end{nalign}
Then
    \begin{nalign}
        \int_{|x|\leqslant 1} |\widecheck{m}(x)| \dx \leqslant 2 \|\widecheck{m}\|_\infty \leqslant C \|m\|_1.  
    \end{nalign}
    Next, by properties of the Fourier transform, the H\"older and Hausdorff-Young inequalities 
    with $p\in[1,1/(1-\alpha))$ and $q = 1- 1/p$, we obtain
    \begin{nalign}
        \int_{|x|>1} |\widecheck{m}(x)| \dx &= \int_{|x|>1} \frac{|x \widecheck{m}(x)|}{|x|}  \dx 
        = \int_{|x|>1} \frac{|\widecheck{(\partial_\xi m)}(x)|}{|x|} \dx \\
        &\leqslant \left\| |x|^{-1} \mathbf{1}_{|x|>1} \right\|_p \left\| \big(\widecheck{\partial_\xi m}\big) \right\|_q 
        \leqslant C \left\| {\partial_\xi m} \right\|_p.
    \end{nalign}

\end{proof}

\begin{lemma}
    \label{lem:UEstim} 
    Let $\alpha \in (0,1)$. 
    For all 
    $$p\in \left(1,\ \frac1{1-\alpha}\right) \quad\text{and}\quad 
    q\in (p,\infty) \quad \text{such that} \quad \frac{1}{q}=\frac{1}{p}-(1-\alpha)
    $$
    there exist a constant $C=C(p,q,\alpha)>0$ such that for arbitrary $f \in \Lp{p} $ we have
    \begin{nalign}\label{eq:est1}
       \| \partial^{-1}_x \Lambda^\alpha f\|_q \leqslant C \|f\|_p.
    \end{nalign}
    Moreover, for another constant $C>0$, we have
    \begin{nalign}\label{eq:est2}
       \| \partial^{-1}_x \Lambda^\alpha f\|_\infty \leqslant C \big( \|f\|_1+ \|f\|_\infty \big).
    \end{nalign}
\end{lemma}

\begin{proof}
    Suppose first that $f \in C_c^\infty(\R)$ and let us show that the integrand in the fractional Laplacian \eqref{eq:frac} satisfies
    \begin{nalign}
        \label{eq:InL1}
        \frac{f(x)-f(y)}{|x-y|^{\alpha+1}} \in L^1(\R\times \R).
    \end{nalign}
    We fix $K>0$ such that $f(x) = 0$ for $|x|>K$, hence, 
    the integrand in \eqref{eq:InL1}
    is equal to  $0$  for $|x|>K$ and $|y|>K$.
Now, we consider the following three cases. 
    
    If $|x|\leqslant K$ and $|y|\leqslant 2K$ (or, analogously,  $|x|\leqslant 2K$ and $|y|\leqslant K$), by the mean value argument, we have
    \begin{nalign}
        \frac{|f(x)-f(y)|}{|y-x|^{\alpha+1}} \leqslant \frac{\|f_x\|_\infty}{|y-x|^\alpha},
    \end{nalign}
    where the right hand side is integrable on the square $[-2K,2K]\times[-2K,2K]$. Indeed, by the direct calculation we have
    \begin{nalign}
        \int_{-2K}^{2K} \int_{-2K}^{2K} \frac{1}{|x-y|^\alpha} \dy \dx &= 2\int_{-2K}^{2K} \int_{-2K}^x \frac{1}{(x-y)^\alpha} \dy \dx  \\
        &=  \frac{2}{1 - \alpha }\int_{-2K}^{2K} (x+2K)^{1-\alpha} \dx =\frac{2(4K)^{2-\alpha}}{(1-\alpha)(2-\alpha)}.
    \end{nalign}
    
    Next, in the case $|x|\leqslant K$ and $|y|>2K$ (or, for $|y|\leqslant K$ and $|x|>2K$) using the triangle inequality we obtain the estimate
    \begin{nalign}
        \frac{|f(x)-f(y)|}{|x-y|^{\alpha+1}} \leqslant \frac{\Ni{f}}{|x-y|^{\alpha+1}} \leqslant \frac{\Ni{f}}{(|y| - |x|)^{\alpha+1}} \leqslant \frac{\Ni{f}}{(|y| - K)^{\alpha+1}}
    \end{nalign}
    where the last term is integrable on the strip $[-K, K] \times \big(\R \setminus [-2K,2K]\big)$.

    Now, we write the fractional Laplacian \eqref{eq:frac} in the equivalent form
    \begin{nalign}
        \Lambda^\alpha f(x)= \int_{\mathbb{R}} \frac{f(x)-f(x+y)}{|y|^{\alpha+1}}  \dy
    \end{nalign}
    and we introduce the primitive function $F(x) = \partial^{-1} f(x)=\int_{-\infty}^x f(x) \dx$. 
    By the intergrability property \eqref{eq:InL1} and the  Fubini theorem, we conclude
    \begin{nalign}
        \label{eq:InFrac}
        \partial_x^{-1} \Lambda^\alpha f(x) = \int_{\mathbb{R}} \int_{-\infty}^x \frac{f(x)-f(x+y)}{|y|^{\alpha+1}} \dx \dy = \int_{\mathbb{R}}  \frac{F(x)-F(x+y)}{|y|^{\alpha+1}} \dy
        =\Lambda^\alpha F(x).
    \end{nalign}

  Now, since $\alpha-1\in(-1,0)$, we may use the Riesz potential operator \cite[Ch.~V]{stein1970singular} given by the formula 
    \begin{nalign}\label{Riesz}
        \Lambda^{\alpha-1 } \varphi (x) = C(\alpha) \int_\R \frac{1}{|x-y|^{\alpha}} \varphi(y) \dy 
    \end{nalign}
    which, by properties of the Fourier transform of the fractional Laplacian, satisfies 
\begin{equation}\label{PLH}
   \partial_x^{-1} \Lambda^\alpha f = \Lambda^\alpha F =\Lambda^{\alpha-1} \mathcal{H} \partial_x F= \Lambda^{\alpha-1} \mathcal{H} f,
\end{equation}
 where $\mathcal{H}$ is the Hilbert transform,~\textit{i.e.}~the pseudo-differential operator with the symbol $-i \, \text{sgn}(\xi)$, thus satisfying 
 $\mathcal{H} \partial_x=\Lambda$. 
 The Hilbert transform is a bounded linear operator on $\Lp{}$ for each $p \in (1,\infty)$, hence, 
inequality \eqref{eq:est1} is an immediate consequence of the Hardy-Littlewood-Sobolev inequality for fractional integration 
\cite[Ch.~V, Thm.~1]{stein1970singular}. 
    We complete the proof of inequality \eqref{lem:UEstim} for arbitrary  $f \in \Lp{p}$   by the usual density argument.

It follows from relation \eqref{PLH}
that 
$$ \widehat{(\partial_x^{-1} \Lambda^\alpha f)}(\xi) = \frac{- i\,\text{sgn}(\xi)}{|\xi|^{1-\alpha}}\widehat{f} (\xi).$$
Hence $ \partial_x^{-1} \Lambda^\alpha f = K_\alpha*f$,
where $K_\alpha=K_\alpha(x)$ is a locally integrable, homogeneous function of degree $-\alpha\in (-1,0)$.
  Using the decomposition
 $K_\alpha=K_\alpha^1+K_\alpha^2$
with 
$$
K_\alpha^1(x) = 
\begin{cases}
    K_\alpha(x), & |x|\leq 1,\\
    0,& |x|\geq 1
\end{cases}
\quad \text{and}\quad K_\alpha^2=K_\alpha-K_\alpha^1
$$
we prove inequality \eqref{eq:est2} immediately 
by using the H\"older inequality .
\end{proof}

\begin{lemma} 
    \label{thm:uFracEst} 
    Fix  $\alpha \in(0,1)$,  $p\in (1,1/\alpha)$ and $q \in (p,\infty)$  such that $1/q = 1/p - \alpha$.
    If $G \in \Lp{p}$,  $\rho \in \Lp{q}$, and
    $u_x = G + \Lambda^\alpha \rho$ in the sense of tempered distributions
    then $\Lambda^{1-\alpha}u \in \Lp{q}$ together with the estimate 
    \begin{nalign}
        \|\Lambda^{1-\alpha} u\|_q \leqslant C\big( \|G\|_p + \|\rho\|_q \big)
    \end{nalign}
    with a constant $C>0$.
\end{lemma}

\begin{proof}
The fractional Laplacian \eqref{eq:frac} and the Riesz potential operator \eqref{Riesz}
satisfy  the equation $\Lambda^{-\alpha}\Lambda^{\alpha} \varphi = \varphi$ in the sense of tempered distributions, see {\it e.g.} \cite[Ch.~V]{stein1970singular}.
    Thus,  the equation $u_x = G - \Lambda^\alpha \rho$ may be written in the form
    \begin{nalign}
        \Lambda^{-\alpha} \partial_x u = \Lambda^{-\alpha}G - \rho.
    \end{nalign} 
    Next, in the same way as in the proof of Lemma \ref{lem:UEstim}, we obtain 
    \begin{nalign}
         \Lambda^{1-\alpha} u = -\mathcal{H} \Lambda^{-\alpha} \partial_x  u = -\left( \mathcal{H} \Lambda^{-\alpha}G - \mathcal{H} \rho \right)
    \end{nalign}
    and the proof is completed by  applying the the Hardy-Littlewood-Sobolev inequality to $\|\Lambda^{-\alpha} G\|_q$.
\end{proof}

For $p \in(1,\infty)$ and $s\in \R$, we recall the classical definition of the Bessel potential space 
\begin{nalign}
    H_p^s\left(\mathbb{R}\right)=\left\{u \in \mathcal{S}^{\prime} : 
    \Big(1+|\xi|^2)^{\frac{s}{2}}  \widehat{u}\Big)^{\widecheck{\ \;  \; }} \in \Lp{}
    \right\},
\end{nalign}
where $\mathcal{S}^{\prime}$ is the space of tempered distributions on $\R$.
This is a Banach space equipped with the norm
\begin{nalign}
    \|u\|_{H_p^s\left(\mathbb{R}\right)} = \left\|
    \left(\left(1+|\xi|^2\right)^{\frac{s}{2}} \widehat{u}\right)^{\widecheck{\ \;  \; }}\right\|_\Lp{}.
\end{nalign}
Then, the Bessel potential space  for a bounded domain $\Omega\subseteq \R$  is defined by 
\begin{nalign}
    H_p^s\left(\Omega\right) =\left\lbrace u\big|_\Omega: u\in H_p^s\left(\mathbb{R}\right) \right\rbrace
\end{nalign}
with the norm
\begin{nalign}
    \label{eq:BesselNorm}
    \|u\|_{H_p^s(\Omega)}=\inf_{{U \in H_p^s\left(\mathbb{R}\right),  U|_{\Omega}=u,}} \|U\|_{H_p^s\left(\mathbb{R}\right)}.
\end{nalign}

\begin{lemma}
    \label{thm:vfracestim}
    Assume that $v\in \Lp{q} $ and $\Lambda^{1-\alpha} v\in \Lp{q}$ for some $q\in (1,\infty)$. Then $v\in H_q^{1-\alpha}\big([-R,R]\big)$ for each $R>0$ together with the estimate
    \begin{nalign}
        \|v\|_{H^{1-\alpha}_q([-R,R])} \leqslant C \left( \|v\|_q + \|\Lambda^{1-\alpha} v\|_q \right).
    \end{nalign}
\end{lemma}

\begin{proof}
By the  theory of Fourier multipliers on $L^p$-spaces (see \textit{e.g.} \cite[Ch IV]{stein1970singular}), we have
 the inequality     
 \begin{nalign}
        \|v\|_{H^{1-\alpha}_q(\R)} \leqslant C \left( \|v\|_q + \|\Lambda^{1-\alpha} v\|_q \right)
    \end{nalign}
which, by the usual cut-off and the definition of the norm in \eqref{eq:BesselNorm}, completes the proof.
\end{proof}

\begin{lemma}[Compact embedding, see \textit{e.g.}~{\cite[Thm. 4.10.1]{triebel1995}}]
    \label{thm:CompactEmbedding}
    Fix $R>0$. For each $q\in(1,\infty)$ and each $s>0$, the embedding $H^s_q\big([-R,R]\big) \subseteq L^q\big([-R,R]\big)$ is compact.
\end{lemma}

\begin{corollary}
    \label{thm:Compact}
    Assume that 
    $\lbrace v^\varepsilon \rbrace_{\varepsilon>0}\subset \Li$ is bounded and 
    $\lbrace \Lambda^{1-\alpha} v^\varepsilon \rbrace_{\varepsilon>0}$ is bounded in $\Lp{q}$ for some $q\in (1,\infty)$ and $\alpha\in (0,1)$. There exist a subsequence 
    $\lbrace v^{\varepsilon_k} \rbrace_{k\in \NN}$ and $v\in \Li$ such that $v^{\varepsilon_k} \to v$ strongly in $L^q\big([-R,R]\big)$ for each $R>0$.
\end{corollary}
\begin{proof} 
 Let $\varphi \in C_c^\infty$ satisfy $0\leqslant \varphi \leqslant 1$,
    $\varphi(x) = 1$ {for} $x\in [-1,1]$ {and} $\text{supp} \,\varphi \subseteq [-2, 2]$. Define $\varphi_R(x)=\varphi(x/R)$.
    By Lemma \ref{thm:vfracestim}, the family  $\lbrace \varphi_R v^\varepsilon \rbrace_{\varepsilon>0} \subset H^{1-\alpha}_q([-R,R])$ is bounded for each $R>0$. By the compactness from Lemma~\ref{thm:CompactEmbedding}, there exist a subsequence $\left\lbrace \varepsilon^k_R\right\rbrace_{k\geqslant 0}$ (depending on $R>0$) and a function $v_R \in L^q([-R,R])$ such that $v^{\varepsilon_R^k} \to v_R$ strongly in $L^q([-R,R])$. 
    Now, the usual diagonal argument gives a subsequence converging in $L^q([-R,R])$ for each $R>0$.
\end{proof}

Finally, we recall two well-known inequalities involving the fractional Laplacian. 
\begin{lemma}[Stroock-Varopoulos inequality,  {\cite[Thm. 2.1]{liskevich11some}}]
    \label{thm:FracStrooqVarop}
    Assume $\alpha \in(0,2)$ and $p\in[1,\infty)$. The inequality
\begin{equation*}
        \int v^p \Lambda^\alpha v \dx \geqslant 4\frac{p}{(p+1)^2} \int_\R \left( \Lambda^\frac{\alpha}{2} \left( v^\frac{p+1}{2}\right) \right)^2 \dx 
   \end{equation*}
    holds true for each sufficiently regular $v=v(x)\geqslant 0$.
\end{lemma}

\begin{lemma}[Gagiliardo--Nirenberg type inequality,  {\cite[Lemma 3.2]{BIK15}}] 
    \label{thm:FracGagiNiren}
    Fix $r>2$ and $q>1$ such that $q<r<2q$. The inequality
    \begin{nalign}\nonumber
        \|v\|_q^{\theta_1} \leqslant C \left\| \Lambda^\frac{\alpha}{2} |v|^\frac{r}{2} \right\|_2^2 \|v\|_1^{\theta_2}
    \end{nalign}
    holds true with
    \begin{nalign}\nonumber
        \theta_1 = \frac{q}{q-1} (r-1 + \alpha) \quad \text{and} \quad  \theta_2 = \theta_1 - r = \frac{r + q(\alpha - 1)}{q-1}.
    \end{nalign}
\end{lemma}

\section{Regularised system}
\subsection{Existence of solutions}
A solution to both  Cauchy problems \eqref{eq:URhoFractional}--\eqref{eq:iniConv} and \eqref{eq:URhoG}--\eqref{eq:ini3eq} is obtained as a limit when $\varepsilon\searrow 0$ of solutions to the regularized system 
\begin{nalign}
    \label{eq:URhoGEps}
    \begin{aligned}
    \rho_t^\varepsilon +(\rho^\varepsilon u^\varepsilon)_x &= \varepsilon \rho^\varepsilon_{xx},   \\ 
    G_t^\varepsilon + (G^\varepsilon u^\varepsilon)_x  &= \varepsilon G^\varepsilon_{xx},   \\
      \partial_x^{-1}( G^\varepsilon+ \Lambda^\alpha\rho^\varepsilon) &= u^\varepsilon 
     \end{aligned}
     \qquad x\in \R, \quad t>0,
\end{nalign}
with~$\varepsilon>0$ and   with the initial conditions 
\begin{nalign}
    \label{eq:iniReg}
    \rho^\varepsilon(x,0) = \rho_0(x) \quad  \text{and}\quad  G^\varepsilon(x,0) = G_0(x).
\end{nalign}

The following  two theorems summarize properties of solutions to this problem.
\begin{theorem}
    \label{thm:GlobalExistEps}
    Let $\alpha \in (0,1]$. For each $\varepsilon>0$ and each $\rho_0, G_0\in \Li\cap \Lp{1}$ satisfying 
     \begin{nalign}
        0 \leqslant G_0(x) \leqslant a \rho_0(x) \quad  \text{for almost  all}\quad   x\in\R \quad \text{and some fixed }  a>0
    \end{nalign}
     there exist a unique  global-in-time mild solution 
     $$
     \rho^\varepsilon, G^\varepsilon \in C\big([0,\infty), L^p(\R)\big)\quad \text{for each} \quad p\in [1,\infty )
     $$ 
     and 
     $u^\varepsilon \in  L^\infty\big(\R\times [0, \infty)\big)$
     to problem~\eqref{eq:URhoGEps}-\eqref{eq:iniReg}.  
This solution has the following  parabolic regularity for each $T>0$
   \begin{equation}\label{rhoG:reg:0}
        \rho^\varepsilon, G^\varepsilon  \in  C\big((0,T],W^{2,p} (\R)\big)   \cap  C^1\big((0,T],L^p (\R)\big)
   \end{equation}
   for each $p\in (1,\infty)$.

    Moreover, it 
    has the following properties for all $t>0${\rm :}
    \begin{itemize}
        \item the mass conservation
        \begin{nalign}\label{Ms}
            M_\rho \equiv \int_\R{\rho_0} (x) \dx = \int_\R{\rho^\varepsilon(x,t)}\dx \quad \text{and} \quad M_G \equiv  \int_\R{G_0(x)}\dx = \int_\R{G^\varepsilon(x,t)}\dx;
        \end{nalign}
        \item the estimate  \label{thm:UEst}
        \begin{nalign}\label{u:infty}
            \|u^\varepsilon( t)\|_\infty \leqslant \|u_0\|_\infty;
        \end{nalign}
        \item the comparison principle 
        \begin{nalign}\label{Pz} 
             0\leqslant G^\varepsilon(x,t) \leqslant a \rho^\varepsilon(x,t) \quad   \text{for all} \quad x\in\R;
        \end{nalign}
        \item  the $L^p$-estimates
        \begin{nalign}\label{Pe}
            \frac{1}{a}\|G^\varepsilon(\cdot, t)\|_p \leqslant\|\rho^\varepsilon(\cdot, t)\|_p \leqslant \|\rho_0\|_p \quad \text{for each} \quad p\in[1,\infty]
        \end{nalign}
        and 
        \begin{nalign} \label{Pe2}
            \frac{1}{a}\|G^\varepsilon(\cdot, t)\|_p \leqslant\|\rho^\varepsilon(\cdot, t)\|_p \leqslant C(p,\alpha) M_\rho^\frac{2+p\alpha}{2p + 2\alpha}  t^{\left(-1+\frac{1}{p}\right)\frac{1}{2+\alpha}} \quad \text{for each} \quad p\in(2,\infty),   
        \end{nalign}
        with the numbers $C(p,\alpha)>0$.
    \end{itemize}
\end{theorem}

\begin{theorem}
    \label{thm:ExistEstimatesEpsAB}
    Under the assumptions of Theorem \ref{thm:GlobalExistEps}, 
    suppose  moreover, that the initial conditions satisfy
        \begin{nalign} 
        0\leqslant b \rho_0(x) \leqslant G_0(x) \leqslant a \rho_0(x) \quad \text{for almost all} \quad x\in \R
    \end{nalign}
    and for given constants $a>0$ and $b>0$. 
    Then, the solution $(u^\varepsilon,\rho^\varepsilon, G^\varepsilon)$ of the regularized problem \eqref{eq:URhoGEps}-\eqref{eq:iniReg}, constructed in 
     Theorem \ref{thm:GlobalExistEps}, satisfies  
    \begin{itemize}
        \item the inequality
        \begin{nalign}\label{ab:est}
           0\leqslant b \rho^\varepsilon(x,t) \leqslant G^\varepsilon(x,t) \leqslant a \rho^\varepsilon(x,t) \quad \text{for almost  all} \quad x\in \R,
           \quad t\geq 0;
        \end{nalign}
            \item the estimates 
    \begin{nalign} \label{est:rho:G:opt}
        \left\|\rho^\varepsilon(\cdot, t)\right\|_p\leqslant C(p) M_\rho^\frac{1}{p}{t^{-1+\frac{1}{p}}} 
        \quad \text{and} \quad  
        \|G^\varepsilon(\cdot, t)\|_p\leqslant C(p)M_\rho^\frac{1}{p}{t^{-1+\frac{1}{p}}} 
    \end{nalign}  
       for all $t>0$, each $p\in [1,\infty]$, and a constant $C(p)$ independent of~$t$.
    \end{itemize}
\end{theorem}

We begin the proof of Theorems \ref{thm:GlobalExistEps} by the following lemma, where we  construct a local-in-time solution to system \eqref{eq:URhoGEps}-\eqref{eq:iniReg} via the  Banach fixed point argument.

\begin{lemma}
    \label{prop:ExistLocalEps}
    Let $\alpha \in (0,1]$. For each $\rho_0, G_0\in \Li\cap \Lp{1}$ and each $\varepsilon>0$, there exist $T>0$ and a unique mild solution 
     of the regularized problem \eqref{eq:URhoGEps}-\eqref{eq:iniReg}{\rm :}
    \begin{equation}\label{spaces:lem}
    \rho^\varepsilon, G^\varepsilon \in C\big([0,T], \Lp1 \cap L^p(\R)\big) \quad \text{for each $p\in (1,\infty)$}
    \end{equation}
    and $u^\varepsilon \in L^\infty \big(\R\times [0,T]\big)$, satisfying weak formulation     
    \begin{nalign} \label{weak:reg}
      -  \int_0^\infty \int_\R \rho^\varepsilon \varphi_t \dx \dt -  \int_\R \rho_0\varphi(\cdot, 0)\dx - \int_0^\infty \int_\R \rho^\varepsilon u^\varepsilon \varphi_x \dx \dt &= \varepsilon\int_0^\infty \int_\R \rho^\varepsilon \varphi_{xx}\dx\dt , \\
       - \int_0^\infty \int_\R G^\varepsilon \psi_t \dx \dt  -  \int_\R G_0\psi(\cdot, 0)\dx - \int_0^\infty \int_\R G^\varepsilon u^\varepsilon \psi_x \dx \dt &= \varepsilon\int_0^\infty \int_\R G^\varepsilon \psi_{xx}\dx\dt, \\
        - \int_0^\infty \int_\R u^\varepsilon \zeta_x \dx \dt - \int_0^\infty \int_\R \rho^\varepsilon \Lambda^\alpha \zeta \dx \dt & = \int_0^\infty \int_\R G^\varepsilon \zeta \dx \dt
    \end{nalign}
    for all $\varphi,\psi,\zeta \in C_c^\infty\big(\R \times [0,\infty)\big)$.
   Moreover the solution has the following parabolic regularity 
   \begin{equation}\label{rhoG:reg}
        \rho^\varepsilon, G^\varepsilon  \in  C\big((0,T],W^{2,p} (\R)\big)   \cap  C^1\big((0,T],L^p (\R)\big)
   \end{equation}
   for each $p\in (1,\infty)$.
\end{lemma}

\begin{proof}
    Here, we consider a mild solution \textit{i.e.}~a solution to the following Duhamel formulation of problem \eqref{eq:URhoGEps}
    \begin{nalign}
        \label{eq:DuhamelEps}
        \rho^\varepsilon(t) &= e^{\varepsilon t\Delta}\rho_0 + \int_0^t \partial_x e^{\varepsilon(t-s) \Delta} (\rho^\varepsilon (s) u^\varepsilon(s)) \ds, \\ 
        G^\varepsilon(t) & = e^{\varepsilon t\Delta}G_0 + \int_0^t \partial_x e^{\varepsilon(t-s) \Delta} (G^\varepsilon (s) u^\varepsilon(s) ) \ds, \\
        u^\varepsilon & = \partial_x^{-1} G^\varepsilon + \partial_x^{-1} \Lambda^\alpha \rho^\varepsilon.
    \end{nalign}
    where the heat semigroup $e^{\varepsilon t \Delta}$ is given as the convolution with the heat kernel $H(x,t) = ({4\pi t)^{-{1}/{2}}}\exp\big(-|x|^2/(4t)\big)$.

    It is  the well-known 
   that for all $\varepsilon>0$, each $1\leq p\leq q\leq \infty$ and all $v\in L^p(\R)$ the properties
\begin{equation*}
    e^{\varepsilon t \Delta} v \in C\big([0,\infty ); L^p(\R)\big), \qquad    \|e^{\varepsilon t \Delta} v\|_p \leqslant  \|v\|_p
\end{equation*}
and
\begin{equation*}
     \|\partial_x e^{\varepsilon t \Delta} v\|_q\leqslant  C (\varepsilon t)^{\frac12 \left(\frac1p -\frac1q\right) -\frac12} \|v\|_p.
\end{equation*}
hold true for all $t>0$. Moreover, by  Lemma \ref{lem:UEstim},
\begin{equation*}
    \|\rho^\varepsilon u^\varepsilon\|_s \leqslant \| \rho^\varepsilon\|_s \|\partial_x^{-1} G^\varepsilon\|_\infty + \| \rho^\varepsilon\|_p \| \partial_x^{-1} \Lambda^\alpha \rho^\varepsilon \|_q
    \leqslant \| \rho^\varepsilon\|_s \| G^\varepsilon\|_1 + C\| \rho^\varepsilon\|_p^2 
\end{equation*}
for exponents $p,q,s \in (1,\infty)$ satisfying 
$$
\frac1s = \frac1p +\frac1q \qquad \text{and} \qquad  \frac{1}{q}=\frac{1}{p}-(1-\alpha).
$$
Here, we recall also the interpolation inequality $\|v\|_q\leq \|v\|_1+\|v\|_p$ with $1\leqslant q\leqslant p\leqslant\infty$.

  Now,  we consider the space 
  $$X_T^p \equiv L^\infty\big([0,T], \Lp1 \cap L^p(\R)\big),$$ with the corresponding norm 
  \begin{equation*}
      \|v\|_{X^p_T} = \sup_{t\in [0,T]}(\No{v(t)} + \|v(t)\|_p).
  \end{equation*}
Now, by direct calculation, we obtain  the estimate
    \begin{nalign}
        \label{z:B1}
        \left\|\int_0^t \partial_x e^{(t-s)\varepsilon \Delta} \big(\rho^\varepsilon(s) u^\varepsilon(s)\big)  \ds\right\|_p
        \leqslant &C(T,\varepsilon) \|\rho^\varepsilon\|_{X_T^p}^2\|G^\varepsilon\|_{X_T^p}^2 \quad \text{for all} \quad t\in [0,T],
    \end{nalign}
    with $\lim\limits_{T\to 0} C(T,\varepsilon)=0$.
   Analogously, we compute the $L^p$-norm of
    $$
    \int_0^t \partial_x e^{(t-s)\varepsilon \Delta} \big(G^\varepsilon(s) u^\varepsilon(s)\big)\ds .
    $$
   It is routine to show, the operator
    \begin{nalign}
        \mathcal{T}\begin{pmatrix}\rho \\ G \end{pmatrix} = \begin{pmatrix} 
        e^{\varepsilon t\Delta}\rho_0 + \int_0^t e^{\varepsilon(t-s) \Delta} (\rho^\varepsilon u^\varepsilon)_x \ds \\
        e^{\varepsilon t\Delta}G_0 + \int_0^t e^{\varepsilon(t-s) \Delta} (G^\varepsilon u^\varepsilon)_x \ds
        \end{pmatrix}
    \end{nalign}
    defines a contraction on a certain ball centered at 
    $\left(e^{\varepsilon t\Delta} \rho_0, e^{\varepsilon t\Delta}G_0 \right)\in X_T^p\times X_T^p$,
     for sufficiently small $T>0$. A standard proof of the uniqueness of these solutions in the space $X^p_T$ is also based on these estimates  and we skip it.

Similarly, it is routine to show that mild solutions are in fact weak solutions.
It follows from estimate \eqref{eq:est2} that $u^\varepsilon \in L^\infty(\R\times [0,T])$. Thus 
the parabolic regularity of $\rho^\varepsilon, G^\varepsilon$ stated in \eqref{rhoG:reg} is classical because
these functions solve the parabolic equations in \eqref{eq:URhoGEps} supplemented with initial  conditions $\rho_0,G_0\in L^p(\R)$.
\end{proof}


Next, we derive the equation for $u^\varepsilon$ in the case of regularized system \eqref{eq:URhoGEps}--\eqref{eq:iniReg}.
\begin{proposition}
    \label{prop:regular} Suppose that a triple $( u^\varepsilon,\rho^\varepsilon, G^\varepsilon)$ is a solution to the regularized problem \eqref{eq:URhoGEps}--\eqref{eq:iniReg}.
    Then $u^\varepsilon=u^\varepsilon(x,t)$ satisfies, in a weak sense,  the  equation 
        \begin{equation}
        \label{eq:URhoEps3}
        \big(u^\varepsilon_t + u^\varepsilon G^\varepsilon + \Lambda^\alpha(\rho^\varepsilon u^\varepsilon) \big)_x = \varepsilon \big(u^\varepsilon_{xx}\big)_x
    \end{equation}
    supplemented with the initial condition
    \begin{equation} \label{eq:URhoEps3:ini}
    u_x(x,0)= G_0(x)+\Lambda^\alpha\rho_0(x).
    \end{equation}
\end{proposition}

\begin{proof} 
  For arbitrary $\zeta\in C^\infty_c(\R\times [0,\infty))$,
we choose $\varphi=\Lambda^\alpha \zeta$ in first
equation of the definition of
weak solutions \eqref{weak:reg}. We
are allowed to do it by
a usual procedure, approximating
$\Lambda^\alpha \zeta$ by a sequence of
smooth compactly supported
functions. Adding the obtained
equation to the second one in \eqref{weak:reg}
with $\psi=\zeta$ and using  third
equation we obtain the weak
formulation of initial value
problem \eqref{eq:URhoEps3}-\eqref{eq:URhoEps3:ini}.
\end{proof}

\begin{lemma}\label{lem:u:est}
   Let  $( u^\varepsilon,\rho^\varepsilon, G^\varepsilon)$ be a local-in-time solution to problem \eqref{eq:URhoGEps}--\eqref{eq:iniReg} constructed in Lemma \ref{prop:ExistLocalEps} with  
   $u_0\equiv \partial^{-1}_x (G_0+\Lambda^\alpha \rho_0)$.
   Then
   \begin{equation}\label{ue;est:123}
       \|u^\varepsilon(t)\|_{\infty}\leqslant\|u_0\|_{\infty} \quad\text{for all}\quad t\in [0,T].
   \end{equation}
\end{lemma}

\begin{proof}
This is more-or-less known reasoning, hence, we
present a draft of the
proof, only.
Mild solutions of the
regularized problem \eqref{eq:URhoGEps}--\eqref{eq:iniReg}
 depend continuously
(in the norms of the space in \eqref{spaces:lem})
upon initial conditions.
Thus, approximating
$\rho_0$ and $G_0$ by smooth smooth functions, 
we may
assume that $\rho^\varepsilon$ and $G^\varepsilon$ and, consequently, also  
$
u^\varepsilon =\partial^{-1}\left(G^\varepsilon+\Lambda^\alpha \rho^\varepsilon\right)
$
are as regular as
required in this proof.
Thus, using equation \eqref{eq:URhoEps3}, the function $u^\varepsilon$
solves the initial value problem
  \begin{nalign}\label{problem:123}
        &u^\varepsilon_t + u^\varepsilon G^\varepsilon + \Lambda^\alpha(\rho^\varepsilon u^\varepsilon)  = \varepsilon u^\varepsilon_{xx}\\
        &u^\varepsilon (x,0)=u_0(x)=\partial_x^{-1}\big(G_0(x)+\Lambda^\alpha \rho_0(x)\big),
 \end{nalign}
 where $u_0\in L^\infty(\R)\cap C(\R)$ for $G_0\in L^1(\R)$ and $\rho_0\in L^1(\R)\cap L^\infty(\R)$, by inequality \eqref{eq:est2}.
Recall also, that
by the definition of the fractional Laplacian \eqref{eq:frac} and relation \eqref{rel:0}, first equation in \eqref{problem:123} can be written in the form 
 \begin{equation}
        \label{eq:URhoEps1}
            u_t^\varepsilon+ u^\varepsilon u^\varepsilon_x - \int_\R  \frac{u^\varepsilon(y,t) - u^\varepsilon(x,t)}{|x-y|^{1+\alpha}} \rho^\varepsilon (y,t) \dy= \varepsilon u^\varepsilon_{xx}.
    \end{equation}

We are going to show, following the reasoning
from \cite[Prop.~2]{DI06}, that
sufficiently regular solutions
of equation \eqref{eq:URhoEps1} satisfy estimate \eqref{ue;est:123}.

Fix $t>0$ and assume
$\sup\limits_{x\in\R} u^\varepsilon (x, t) > 0$ (the complementary inequality can be handled analogously).
Let $\{x_n\}_{n=1}^\infty\subset \R$ be
a sequence such that
$$
\lim_{n\to\infty} u^\varepsilon(x_n,t) = \sup_{x\in\R} u(x,t).
$$
By \cite[Thm.~2]{DI06}, we have
\begin{equation}\label{rel1}
\lim_{n\to\infty} u^\varepsilon_x(x_n,t) =0.
\end{equation}
Moreover, repeating arguments
from \cite[Proof of Thm.~2]{DI06} 
we obtain 
\begin{equation}\label{rel2}
\liminf_{n\to\infty} u^\varepsilon_{xx}(x_n,t) \leq 0
\end{equation}
and 
\begin{equation}\label{rel3}
\liminf_{n\to\infty}
 \int_\R  \frac{u^\varepsilon(y,t) - u^\varepsilon(x_n,t)}{|x-y|^{1+\alpha}} \rho^\varepsilon (y,t) \dy
\leq 0.
\end{equation}
Fix $a>0$.
For $u_{tt}$ bounded on $\R\times [a,T]$,
for all $t\in (a,T)$ and $\tau \in (a,t)$,
by the Taylor expansion,
we obtain
\begin{nalign}
u(x_n,t) &\leq u(x_n, t-\tau) +\tau u_t(x_n, t)+C\tau^2\\
&\leq \sup_{x\in\R} u(x, t-\tau) +\tau u_t(x_n, t)+C\tau^2.
\end{nalign}
Using equation \eqref{eq:URhoEps1} for $u_t(x_n,t)$, 
passing to  $\liminf$, and applying 
 relations \eqref{rel1}--\eqref{rel3}, 
 we
obtain the inequality
$$
 \sup_{x\in\R} u(x, t) \leq \sup_{x\in\R} u(x, t-\tau)+ C\tau^2.
$$
or equivalently
$$
\frac{\sup_{x\in\R} u(x, t) - \sup_{x\in\R} u(x, t-\tau)}{\tau}\leq C\tau.
$$
It is easy to prove, that the function
$h(t)=\sup\limits_{x\in \R} u(x,t)$ is Lipschitz continuous on $[a,T]$
when $u_t$ is bounded on $\R\times [a,T]$,  and herefore the last inequality implies
$$
\frac{d}{dt} \left( \sup_{x\in \R} u(x,t)\right)\leq 0 \quad \text{almost everywhere in} \quad t\in(a,T).
$$
Since $a\in (0,T)$ is  arbitrary, by 
continuity of $u^\varepsilon(x,t)$ we obtain
$$
 \sup_{x\in\R} u(x, t) \leq \sup_{x\in\R} u_0(x).
$$
The proof with ``sup'' replaced
by ``inf'' is analogous.
\end{proof}

\begin{proof}[Proof of Theorem \ref{thm:GlobalExistEps}]
    Let $(u^\varepsilon, \rho^\varepsilon, G^\varepsilon)$ be the local-in-time solution constructed in 
    Lemma \ref{prop:ExistLocalEps}. 
    First, let us show, that the properties of this solution 
    stated in Theorem \ref{thm:GlobalExistEps} hold true
    for all $t \in [0, T]$. 

    Since 
    the heat kernel $H(x,t) = (4\pi t)^{-{1}/{2}}\exp\big(-|x|^2/(4t)\big)$ satisfies
    \begin{nalign}\int_\R H(x,t) \dx = 1 \quad  \text{and} \quad \int_\R  H_x(x,t) \dx = 0, \end{nalign}  
    integrating the Duhamell equations in ~\eqref{eq:DuhamelEps}
    and applying the Fubini theorem,  proves that the masses  $M_\rho$ and  $M_G$ given by formulas \eqref{Ms}
    are independent of time. 
    
    Inequality \eqref{u:infty} is proved in Lemma \ref{lem:u:est}.

    The function  $w\equiv a \rho^\varepsilon - G^\varepsilon$ satisfies the equation
    \begin{equation}\label{eq:w}
        w_t+(u^\varepsilon w)_x=\varepsilon w_{xx}, 
    \end{equation}
    where $u^\varepsilon$ is a bounded and continuous function. Hence,  inequalities \eqref{Pz} are the immediate consequence 
    of the fact, that solutions to parabolic equation \eqref{eq:w} remain nonnegative if initial conditions are so. 

    Notice, that estimate \eqref{Pe} for $p = 1$  follows directly from the mass conservation~\eqref{Ms} because $\rho^\varepsilon \geqslant 0$. In order to proceed for $p\in(1,\infty)$, we multiply first equation in~\eqref{eq:URhoGEps} by $p \left(\rho^\varepsilon\right)^{p-1}$ and integrate:
    \begin{nalign}
        \frac{d}{dt}\int_\R \left(\rho^\varepsilon\right)^p \dx + p \int_\R (\rho^\varepsilon u^\varepsilon)_x \left(\rho^\varepsilon\right)^{p-1} \dx = {\varepsilon p} \int_\R \rho^\varepsilon_{xx} \left(\rho^\varepsilon\right)^{p-1} \dx.
    \end{nalign}
    Next, we integrate the middle term twice by parts
    \begin{nalign}
        p\int_\R (\rho^\varepsilon u^\varepsilon)_x \left(\rho^\varepsilon\right)^{p-1} \dx &= -(p-1)p \int_\R \rho^\varepsilon u^\varepsilon \left(\rho^\varepsilon\right)^{p-2}\rho^\varepsilon_x \dx \\  &= -(p-1) \int_\R  u^\varepsilon \big(\left(\rho^\varepsilon\right)^p\big)_x \dx = (p-1) \int_\R  u^\varepsilon_x \left(\rho^\varepsilon\right)^p \dx
    \end{nalign}
    and we substitute the relation $u_x^\varepsilon = G^\varepsilon + \Lambda^\alpha\rho^\varepsilon$ to obtain
    \begin{nalign} 
        \label{z:B3}
        \frac{d}{dt}\int_\R \left(\rho^\varepsilon\right)^p \dx = &- (p-1) \int_\R \left(\rho^\varepsilon\right)^p \Lambda^\alpha \rho^\varepsilon \dx \\ &- (p-1) \int_\R  \left(\rho^\varepsilon\right)^p G^\varepsilon \dx -  \varepsilon p(p-1) \int_\R (\rho^\varepsilon_x)^2 \left(\rho^\varepsilon\right)^{p-2} \dx.
    \end{nalign}
    Since $\rho^\varepsilon\geqslant 0$ and $G^\varepsilon\geqslant 0$,  the last two terms on the right hand side are nonpositive. 
    We estimate the remaining term $(p-1)\int_\R \left(\rho^\varepsilon\right)^p \Lambda^\alpha \rho^\varepsilon \dx$ by using 
    the Strook-Varopoulos inequality from   
    Lemma \ref{thm:FracStrooqVarop} to obtain 
    \begin{nalign}
        \label{eq:RhoDeriv}
        \frac{d}{dt}\int_\R \left(\rho^\varepsilon\right)^p \dx  \leqslant -4\frac{p(p-1)}{(p+1)^2}\int_\R \left( \Lambda^\frac{\alpha}{2} (\rho^\varepsilon)^\frac{p+1}{2} \right)^2 \dx, 
    \end{nalign}
which implies  the estimates  $\left\|\rho^{\varepsilon}(t)\right\|_p \leqslant \left\|\rho_0\right\|_p$ 
as well as (by \eqref{Pz})
 $\left\|G^{\varepsilon}(t)\right\|_p \leqslant a\left\|\rho_0\right\|_p$
for each $p\in (1,\infty)$ and all $t\in [0,T]$.
The remaining $L^\infty$-estimates in inequalities \eqref{Pe} are obtained by passing  to the limit with $p\to\infty$.

    Next, we combine Lemma \ref{thm:FracGagiNiren} with $r=p+1$ and $q=p$ together with the mass (hence, the $L^1$-norm) conservation \eqref{Ms}:
    \begin{nalign}
        \label{eq:RhoBelow}
        \int_\R \left( \Lambda^\frac{\alpha}{2} \rho^\varepsilon)^\frac{p+1}{2} \right)^2 \geqslant C(p,\alpha) 
        \frac{1}{M^\frac{1+p\alpha}{p-1}_\rho} \left(\int_\R  \left(\rho^\varepsilon\right)^p \dx\right)^\frac{p(p+\alpha)}{p-1}.
    \end{nalign}
    We apply inequality \eqref{eq:RhoBelow} in
      \eqref{eq:RhoDeriv}  to obtain the estimate 
    \begin{nalign}
         \frac{d}{dt}\int_\R \left(\rho^\varepsilon\right)^p \dx  \leqslant - C(p,\alpha) 
        \frac{1}{M^\frac{1+p\alpha}{p-1}_\rho} \left(\int_\R  \left(\rho^\varepsilon\right)^p \dx\right)^\frac{p(p+\alpha)}{p-1}
    \end{nalign}
   which is valid for each $p\in(1,\infty)$.  By solving this differential inequality, we prove  decay estimate \eqref{Pe2}.

To conclude the proof, we notice, that 
 the norms $\|\rho^\varepsilon(t)\|_p$ and $\|G^\varepsilon(t)\|_p$ are {\it a priori} bounded in $t>0$ for each $p\in [1,\infty]$ ({\it cf.} estimates \eqref{Pe}), hence the solution $(\rho^\varepsilon(x,t),G^\varepsilon(x,t)$ can be extended to all $t>0$. In this way,  relations 
\eqref{Ms}-\eqref{Pe2} hold true for  all $t>0$. 
\end{proof}

\begin{proof}[Proof of Theorem \ref{thm:ExistEstimatesEpsAB}]
   Inequalities \eqref{ab:est} are  immediate consequence of the comparison principle, in the same way as in the proof of estimates 
   \eqref{Pz}.
   
   It follows from  equation  \eqref{z:B3}, the Strook-Varopoulos inequality (Lemma \ref{thm:FracStrooqVarop}) and the inequality  $G(x,t)\geqslant b \rho(x,t)$ that 
    \begin{nalign}
         \frac{d}{dt}\int_\R \left(\rho^\varepsilon\right)^p \dx  \leqslant  - b(p-1) \int_\R  \left(\rho^\varepsilon\right)^{p+1} \dx.
    \end{nalign}
    By the H\"older inequality and the mass conservation \eqref{Ms}, we obtain
    \begin{nalign}
        \int_\R \left(\rho^\varepsilon\right)^{p} \dx= \int_\R  \left(\rho^\varepsilon\right)^\frac{1}{p}\left(\rho^\varepsilon\right)^{p-\frac{1}{p}}\dx \leqslant M_\rho^\frac{1}{p}\left(\int_\R \left(\rho^\varepsilon\right)^{p+1}\dx\right)^\frac{p-1}{p}
    \end{nalign}
    and therefore 
    \begin{nalign}
        \frac{d}{dt}\int_\R \left(\rho^\varepsilon\right)^p \dx  \leqslant  - b(p-1) M_\rho^\frac{1}{1-p}\left( \int_\R  \left(\rho^\varepsilon\right)^{p} \dx \right)^\frac{p}{p-1}. 
    \end{nalign} 
    Thus, solving this differential inequality we obtain  inequality \eqref{est:rho:G:opt} for $p\in[1,\, \infty)$ with $C(p) = b^{-\frac{p-1}{p}}$. 
    Since $\sup\limits_{p\in(1,\infty)} C(p)<\infty$, we also have this inequality for $p=\infty$. 
    The corresponding estimate for $G$ in \eqref{est:rho:G:opt} is a consequence of inequality~\eqref{ab:est}.
\end{proof}

\section{Compactness estimates and passage to the limit} 
By estimates \eqref{u:infty} and \eqref{Pe},
the  sequence of solutions $(\rho^\varepsilon, G^\varepsilon, u^\varepsilon)$ 
of the regularized problem \eqref{eq:URhoGEps}--\eqref{eq:iniReg}
is uniformly bounded with respect to $\varepsilon>0$ in the following spaces
\begin{nalign}\label{spaces}
    \lbrace \rho^\varepsilon \rbrace_{\varepsilon>0} &\subseteq C \big([0,\infty),\, \Lp{} \big) \quad \text{for each} \quad p\in[1,\infty], \\
    \lbrace G^\varepsilon \rbrace_{\varepsilon>0} & \subseteq  C \big([0,\infty), \, \Lp{} \big) \quad \text{for each} \quad p\in[1,\infty].\\ 
\end{nalign}
 Moreover,  the sequence $u^\varepsilon$ is uniformly bounded with respect to $\varepsilon>0$ and $R>0$ in the space 
\begin{nalign}
    \label{eq:ULiHa}
    \lbrace u^\varepsilon \rbrace_{\varepsilon>0} &\subseteq L^\infty\big( [0,\infty), \, H^{1-\alpha}_q([-R,R]) \big) \quad \text{for each} \quad q\in \big(1, \infty\big) 
\end{nalign}
which  is a straightforward consequence of Lemmas \ref{thm:uFracEst} and \ref{thm:vfracestim} combined with estimates \eqref{u:infty} and \eqref{Pe}.

\begin{lemma}
    \label{thm:UTEst}
    The sequence $\lbrace \partial_t u^\varepsilon \rbrace_{\varepsilon>0}$ is bounded in $L^\infty\big( [0,\infty), W^{2,1}(\R)^* \big)$,
    where we denote by $W^{2,1}(\R)^*$  the dual space  to the Sobolev space $W^{2,1}(\R)$. 
\end{lemma}

\begin{proof}
    By Proposition \ref{prop:regular}, for arbitrary $T>0$ and $\varphi \in L^\infty\big( [0,T), W^{2,1}(\R) \big)$, we have
    \begin{nalign}
        \int_{\R} u_t^\varepsilon \varphi \dx 
        & = - \int_\R \rho^\varepsilon u^\varepsilon \Lambda^\alpha \varphi \dx  - \int u^\varepsilon G^\varepsilon\varphi \dx + \varepsilon \int_\R u^\varepsilon \varphi_{xx} \dx. 
    \end{nalign}
    Since, the functions $(\rho^\varepsilon, G^\varepsilon, u^\varepsilon)$ are uniformly bounded in $L^\infty$ with respect to $\varepsilon>0$ ({\it c.f.}~inequalities \eqref{u:infty} and \eqref{Pe}),  by  Lemma \ref{lem:FracEstim}, we obtain the estimate
    \begin{nalign}
        \sup_{t>0}\left|\int_{\R} u_t \varphi \dx \right| &\leqslant C \sup_{t>0} \big( \|\varphi(\cdot,t)\|_1 + \|\Lambda^\alpha\varphi(\cdot,t)\|_1 + \|\varphi_{xx}(\cdot,t)\|_1 \big) \\
        & \leqslant C \sup_{t>0} \|\varphi(\cdot,t)\|_{2,1},
    \end{nalign}
    with a constant $C = C\big(\Ni{\rho_0},\Ni{G_0},\Ni{u_0}\big) >0$ independent of $t>0$ and of $\varepsilon>0$.
\end{proof}

\begin{proof}[Proof of Theorem \ref{thm:GlobExist}] 
We pass to the limit with $\varepsilon\searrow 0$ in
 the  weak formulation \eqref{weak:reg} of  regularized problem \eqref{eq:URhoGEps}--\eqref{eq:iniReg}.
    Here, we also  apply the  weak formulation of problem \eqref{problem:123} for $u^\varepsilon$ supplemented with an initial datum $u_0\in \Li$:
 \begin{nalign}
    \label{weak:reg2}
          -  \int_0^\infty \int_\R u^\varepsilon \varphi_t \dx \dt & - \int_\R u_0 \varphi(\cdot,0) \dx 
            + \int_0^\infty \int_\R  u^\varepsilon G^\varepsilon \varphi \dx \dt   \\ 
            &+ \int_0^\infty \int_\R \rho^\varepsilon u^\varepsilon  \Lambda^\alpha\varphi \dx \dt = 
            \varepsilon \int_0^\infty \int_\R u^\varepsilon \varphi_{xx} \dx \dt.
\end{nalign}

Since the families 
$\lbrace \rho^\varepsilon\rbrace_{\varepsilon>0},  \lbrace G^\varepsilon\rbrace_{\varepsilon>0}, \lbrace u^\varepsilon \rbrace_{\varepsilon>0}$  are
bounded in the spaces  \eqref{spaces}, 
we may choose a subsequence $\{\varepsilon_k\}^\infty_{k=1}$ such that  
\begin{nalign}\label{weak-star:2}
    \begin{matrix}
        \rho^{\varepsilon_k} \to \rho, \\
        G^{\varepsilon_k}\to G,
    \end{matrix}
    & & & \text{weakly in} \; L^r \big([0,T],\, \Lp{} \big),
\end{nalign}
for all $r,p\in (1,\infty)$ and each $T>0$ (notice that if $\rho^{\varepsilon_k} \to \rho$ weakly in $L^r\big([0,T], \Lp{} \big)$ and in $L^{\tilde{r}}\big([0,T], \Lp{\tilde{p}} \big)$ then then it converges weakly in $L^z\big([0,T], \Lp{s} \big)$ for all $z \in (r,\tilde{r})$ and $s \in (p, \tilde{p})$).

To deal with  nonlinear terms, we also show
    \begin{nalign} \label{u:conv:strong}
         u^{\varepsilon_k} \to u\quad  \text{strongly in} \quad C\big([0,T], L^q([-R,R])\big),
    \end{nalign}
    for each  $R>0$, $T>0$, and  $q\in (1,\infty )$.
    Indeed, we consider three spaces 
    \begin{nalign}
       X  = H^{1-\alpha}_q([-R,R]), \qquad   Y  = L^q([-R,R]), \qquad Z  = W^{2,1}\big([-R,R]\big)^\ast,
    \end{nalign} 
    where embedding $X\subseteq Y$ is compact by Lemma \ref{thm:CompactEmbedding} and Corollary \ref{thm:Compact}. By the classical Sobolev embedding the injection 
    $W^{2,1}\big([-R,R]\big) \subseteq  L^q\big([-R,R]\big)$
    is continuous and hence the dual injection  $Y \subseteq Z$ is continuous. Thus, the Aubin-Lions-Simon Lemma~\cite{simon1986compact}, states that the embedding
    \begin{nalign} \label{imb2}
        \left\lbrace v\in L^\infty\big([0,T], X \big)\,:\, \partial_t v \in L^\infty\big([0,T],Z\big) \right\rbrace  \subseteq C\big([0,T], Y)\big)
    \end{nalign}
    is compact. In our case, by Lemma \ref{thm:UTEst} and property  \eqref{eq:ULiHa}, the sequence $\{u^{\varepsilon_k}\}_k$ is bounded in the space
    on the left-hand side of relation \eqref{imb2}.
    Consequently, by a standard diagonal argument,  convergence \eqref{u:conv:strong} holds true for each $T>0$, $R>0$, and $q\in (1,\infty)$.
    
    Now, by the weak convergence \eqref{weak-star:2} and the strong convergence \eqref{u:conv:strong}, we may pass immediately to the limit in the  linear terms   
    in equations \eqref{weak:reg} and \eqref{weak:reg2}, 
    namely, these containing only one  function either $\rho^\varepsilon$ or $G^\varepsilon$ or $u^\varepsilon$ multiplied by either $\varphi$ or $\varphi_t$ or $\varphi_x$ or $\varphi_{xx}$ with arbitrary $\varphi\in C_c^\infty\big(\R \times [0,\infty)\big)$.
    Similarly, 
    $$
    \int_0^\infty \int_\R \rho^{\varepsilon_k} \Lambda^\alpha \varphi \dx \dt \to  \int_0^\infty \int_\R \rho \Lambda^\alpha \varphi \dx \dt,
    $$
    because $ \Lambda^\alpha \varphi \in C \big([0,\infty); L^1(\R)\cap L^\infty (\R) \big)$ (see {\it e.g.} Lemma \ref{lem:FracEstim})
    and is compactly supported in time.
    Thus, for example, by third equation in system \eqref{weak:reg},
    we obtain the limit functions in \eqref{weak-star:2} and \eqref{u:conv:strong} satisfy the equation $u_x-\Lambda^\alpha \rho=G$ in the sense of distributions. 

Next, we combine the weak convergence of $\rho^{\varepsilon_k}$ and $G^{\varepsilon_k}$ from \eqref{weak-star:2}
with the strong convergence of  $u^{\varepsilon_k}$ in \eqref{u:conv:strong} and we take into account the compact support of $\varphi$ 
to pass to the following limits
\begin{nalign}
     \int_0^\infty \int_\R \rho^{\varepsilon_k} u^{\varepsilon_k} \varphi_x \dx \dt &\to   \int_0^\infty \int_\R \rho u\varphi_x \dx \dt,\\
        \int_0^\infty \int_\R G^{\varepsilon_k} u^{\varepsilon_k} \varphi_x \dx \dt &\to  \int_0^\infty \int_\R G u\varphi_x \dx \dt,\\
           \int_0^\infty \int_\R G^{\varepsilon_k} u^{\varepsilon_k} \varphi\dx \dt &\to   \int_0^\infty \int_\R G u\varphi \dx \dt.
\end{nalign}
To handle the remaining integral in equation \eqref{weak:reg2} with the integrand $u^\varepsilon\rho  \Lambda^\alpha\varphi$, we proceed analogously 
using the mentioned above properties of $\Lambda^\alpha\varphi$ with $\varphi \in C_c^\infty\big(\R \times [0,\infty)\big)$.

Next, we show that  properties \eqref{u:infty}-\eqref{Pe}
of solutions to the regularized system remains valid after passing to the limit $\varepsilon \to 0$.
    
By convergence \eqref{u:conv:strong} (after choosing a subsequence), we have $u(x,t)^\varepsilon \to u(x,t)$ almost everywhere in $(x,t)\in \R\times (0,\infty)$.
Hence, we may  pass to the limit in the inequality  ({\it cf.} \eqref{u:infty})
\begin{nalign}
|u^\varepsilon(x,t)|\leqslant\|u_0\|_\infty \quad\text{a.e. in} \quad x\in\R, t>0
\end{nalign}
to obtain estimate \eqref{u:inf:0}.

Inequalities \eqref{Garho:0} are deduced immediately after passing to the limit by using the weak convergence \eqref{weak-star:2}
in the following inequalities (resulting from estimates \eqref{Pz})
$$
\int_0^\infty \int_\R  G^\varepsilon\varphi  \dx\dt\geqslant 0 \quad \text{and}\quad 
\int_0^\infty \int_\R \big(a \rho^\varepsilon - G^\varepsilon \big)\varphi  \dx\dt\geqslant 0,
$$
 which are valid for each {\it nonnegative} $\varphi \in C_c^\infty\big(\R \times [0,\infty)\big)$.

By second inequality in \eqref{Pe}, for each $\varphi \in C_c^\infty\big( \R \times [0,\infty)\big)$ and each $p\in (1,\infty)$, with $p' = p/(p-1)$ we have
\begin{nalign}
    \left| \int_0^\infty \int_\mathbb{R} \rho^\varepsilon \varphi \dx \dt \right| \leqslant \|\rho_0\|_p \int_0^\infty \|\varphi(\cdot, t)\|_{p'} \dt.
\end{nalign}
Passing to the weak limit \eqref{weak-star:2} we obtain analogous inequality with $\rho^\varepsilon$ replaced by $\rho$ which means, that $\rho$ defines a bounded linear functional on $L^1\big( (0,\infty), L^{p'}(\R)\big)$. Thus, 
\begin{nalign}
    \|\rho\|_{\left(L^1( (0,\infty), L^{p'}(\R))\right)^*} =     \|\rho\|_{L^\infty\big( (0,\infty), L^{p}(\R)\big)} \leqslant \|\rho_0\|_p, 
\end{nalign}
which gives inequalities \eqref{eq:Lp:est} (the first one, follows from \eqref{Garho:0}).
We prove second inequality in \eqref{eq:DecEstim} analogously
by considering 
$$ \widetilde{\rho^\varepsilon}(x,t)=t^{-\left(-1+\frac{1}{p}\right)\frac{1}{2+\alpha}} \rho^\varepsilon(x,t)
$$
and using  estimate \eqref{Pe2}.

\end{proof}

\begin{proof}[Proof of Theorem \ref{thm:DecayEstiamtes}]       
  Here, the reasoning is completely analogous to the proof of  Theorem \ref{thm:GlobExist} 
  and is based on properties of solutions to the regularized problem proved in Theorem \ref{thm:ExistEstimatesEpsAB}.
\end{proof}

\section{Asymptotic behavior of solutions}
In the following, we denote by  $(u,\rho, G)$  a solution to both  problems  \eqref{eq:URhoFractional}--\eqref{eq:iniConv} and \eqref{eq:URhoG}--\eqref{eq:ini3eq} constructed in Theorem \ref{thm:GlobExist} and satisfying decay estimates \eqref{eq:RhoGDecayEstim} from Theorem \ref{thm:DecayEstiamtes}.
For each $\lambda>0$, we rescale this solution in the following way
\begin{align}\label{recal:sol}
    \rho^\lambda(x,t) \equiv \lambda \rho(\lambda x,\lambda t), \quad G^\lambda(x,t) \equiv  \lambda G(\lambda x,\lambda t), \quad u^\lambda(x,t) \equiv  u(\lambda x,\lambda t).
\end{align}

\begin{lemma}
 \label{thm:LambdaEstim1}
   For each $p\in (1,\infty)$, there exist a number $C= C(p,M_\rho)>0$  such that for all $\lambda>0$ and all $t>0$ the following inequalities hold true
    \begin{nalign}
        \|\rho^\lambda(\cdot,t)\|_p \leqslant C t^{-1+1/p}, \quad \|G^\lambda(\cdot,t)\|_p \leqslant C t^{-1+1/p}, \quad
        \|u^\lambda(\cdot,t)\|_\infty \leqslant \|u_0\|_\infty.
    \end{nalign}
\end{lemma}

\begin{proof}
    By decay estimate  \eqref{eq:RhoGDecayEstim}, changing variables, we  obtain for all $t>0$
    $$
        \|\rho^\lambda(\cdot,t)\|_p \leqslant \|\lambda \rho(\lambda \cdot,\lambda t)\|_p 
        \leqslant C (\lambda t)^{-1 + 1/p}\lambda^{-1 + 1/p} = C t^{-1 + 1/p}.
   $$
    The proofs of the reaming  inequalities are analogous.
\end{proof}

\begin{lemma}
    \label{thm:LambdaEstim2}
    Let $\alpha \in(0,1)$ and $q\in \big( 1/\alpha, \, \infty\big)$. For all $t_1, t_2\in (0,\infty)$ and each $R>0$ the sequence $\lbrace u^\lambda \rbrace_{\lambda \geqslant 1}$ is bounded in $L^\infty\big( [t_1, t_2], H^{1-\alpha}_q[-R,R] \big)$.
\end{lemma}

\begin{proof}
    Let $p = q/(q-1)$. Applying Lemmas \ref{thm:uFracEst} and \ref{thm:vfracestim}   we have
    \begin{nalign}\nonumber
        \|u^\lambda(t)\|_{H^{1-\alpha}_q([-R,R])} &\leqslant C \left(\| u^\lambda (t)\|_q+\|\Lambda^{1-\alpha} u^\lambda (t) \|_q\right)  \\
        &\leqslant C\big(\| u^\lambda (t)\|_q+ \|G^\lambda(t)\|_p + \|\rho^\lambda(t)\|_q\big).
    \end{nalign}
    By Lemma \ref{thm:LambdaEstim1} the right-hand side is uniformly bounded with respect to $\lambda>0$ and $t \in [t_1,t_2]$ for all $t_1,t_2\in(0,\infty)$.  
\end{proof}

\begin{lemma}\label{lem:uT:lambda}
    The sequence $\{\partial_t u^\lambda\}_{\lambda\geq 1}$ is bounded in the space $L^\infty \big([t_1,t_2] , W^{1,2}(\R)^*   \big)$ for all $t_1,t_2 \in (0,\infty)$.
\end{lemma}
\begin{proof}
    Here, it suffices to repeat the reasoning from the proof of Lemma \ref{thm:UTEst}.
\end{proof}

We are in a position to prove the main result of this work on the asymptotic behavior of the rescaled solutions  \eqref{recal:sol}.

\begin{proof}[Proof of Theorem \ref{thm:asymp}]
Introducing the  new variables $y = x/\lambda$ and $s = t/\lambda$,
the rescaled functions \eqref{recal:sol} satisfy  the following counterparts of equations  \eqref{eq:URhoFractional} and \eqref{eq:URhoG}
\begin{nalign}
    \label{eq:URho2eqLambda}
    \rho^\lambda_s + (u^\lambda \rho^\lambda)_y &= 0, \\
    u^\lambda_s +u^\lambda (u_y^\lambda - \lambda^{-\alpha} \Lambda^\alpha \rho^\lambda) +\lambda^{-\alpha} \Lambda^\alpha (u^\lambda \rho^\lambda)&=0,
\end{nalign}
as well as
\begin{nalign}
    \label{eq:URhoGLambda}
    \rho^\lambda_s +  (u^\lambda \rho^\lambda)_y &= 0, \\
    G^\lambda_s + (u^\lambda G^\lambda)_y &= 0, \\
   u^\lambda_y - \lambda^{-\alpha} \Lambda^\alpha \rho^\lambda &=  G^\lambda.
\end{nalign}
Here, we impose   the initial conditions 
\begin{align}
    \rho_0^\lambda(x)= \lambda \rho_0(\lambda x), \quad 
    G_0^\lambda(x) =  \lambda G_0(\lambda x), 
    \quad u_0^\lambda(x) = u_0(\lambda x)
\end{align}
and we understand the corresponding initial value problems 
 in the weak sense:
\begin{nalign}
    \label{eq:URho2eqLambdaWeak}
   - \int_0^\infty \int_\R \rho^\lambda \varphi_s \dy \ds -\int_\R \rho^\lambda_0 \varphi(\cdot, 0) \dy &-  \int_0^\infty \int_\R \rho^\lambda u^\lambda \varphi_y \dy \ds = 0, \\
   - \int_0^\infty \int_\R u^\lambda \psi_s \dy \ds  -\int_\R u^\lambda_0 \psi (\cdot, 0) \dy &+\lambda^{-\alpha} \int_0^\infty \int_\R  u^\lambda\rho^\lambda  \Lambda^\alpha\psi \dy \ds \\ &+ \int_0^\infty \int_\R  u^\lambda(u_y^\lambda - \lambda^{-\alpha}\Lambda^\alpha \rho^\lambda)\psi \dy \ds= 0
\end{nalign}        
as well as
\begin{nalign}
    \label{eq:UrhoGLambdaWeak}
   - \int_0^\infty \int_\R \rho^\lambda \varphi_s \dy \ds -\int_\R \rho^\lambda_0 \varphi(\cdot, 0) \dy  - \int_0^\infty \int_\R \rho^\lambda u^\lambda \varphi_y \dy \ds &= 0, \\
   - \int_0^\infty \int_\R G^\lambda \psi_s \dy \ds -\int_\R G^\lambda_0 \psi(\cdot, 0) \dy - \int_0^\infty \int_\R G^\lambda u^\lambda \psi_y \dy \ds &= 0, \\
- \int_0^\infty \int_\R u^\lambda \zeta_y \dy \ds   -\lambda^{-\alpha}\int_0^\infty \int_\R \rho^\lambda \Lambda^\alpha \zeta \dy \ds  - \int_0^\infty \int_\R G^\lambda \zeta \dy \ds & = 0,
\end{nalign}
for all $\varphi, \psi,\zeta \in C_c^\infty\big(\R \times [0,\infty)\big)$.

    By Lemmas \ref{thm:LambdaEstim1}-\ref{lem:uT:lambda} 
    combined with Lemma \ref{thm:CompactEmbedding} and Corollary \ref{thm:Compact} and by the Aubin-Lions-Simon argument applied 
    in the same way as in the proof of the convergence in \eqref{weak-star:2}-\eqref{u:conv:strong}, we extract 
   a subsequence $\lambda_k \to \infty$ and we find the tuple $(\overline{u}, \overline{\rho}, \overline{G})$ such that 
    \begin{nalign}
        \label{eq:WeakLambdaConvergence2}
        \begin{matrix}
            \rho^{\lambda_k} \to \overline{\rho} \\
             G^{\lambda_k} \to \overline{G}
        \end{matrix}& & & \text{weakly in} \; \; L^r\big( [t_1, t_2], \Lp{} \big),  \\
        u^{\lambda_k} \to \overline{u}& & & \text{strongly in} \; \; C\big( [t_1, t_2], L^q([-R,R]) \big),
    \end{nalign} 
for each 
 $t_1,t_2\in (0,\infty)$, every $R>0$ and $r,p\in(1,\infty)$ and $q\in\big( 1/\alpha, \infty\big)$.

    Now, we pass to the limit in equations \eqref{eq:URho2eqLambdaWeak} and \eqref{eq:UrhoGLambdaWeak} and we begin with the terms containing initial conditions.
     Since $\rho_0, G_0\in \Lp{1}$ and $\varphi(\cdot,0)\in C^\infty_c(\R)$ we may change variables  and apply the Lebesgue  dominated convergence theorem to conclude
    \begin{nalign}
        \int_\R \rho_0^{\lambda_k}(x) \varphi(x,0) \dx = \int_\R \rho_0(y) \varphi \left(\frac{y}{\lambda_k},0\right) \dy \to \varphi(0,0) M_\rho
    \end{nalign}
    and similarly
    \begin{nalign}
        \int_\R G_0^{\lambda_k}(x) \psi(x,0) \dx \to \psi(0,0) M_G,
    \end{nalign}
   with $M_\rho = \int_\mathbb{R} \rho_0(x) \dx$ and $M_G = \int_\mathbb{R} G_0 (x) \dx$. 
Moreover, we recall 
\begin{nalign}u_0=\partial^{-1}_x G_0+\partial^{-1}_x \Lambda^\alpha\rho_0,\end{nalign}
where 
$$
\left(\partial^{-1}_x G_0\right)(\lambda_k x) = \int_{-\infty}^{\lambda_k x} G_0(y)\dy \to 
\begin{cases}
    M_G& \text{for} \;\; x>0,\\
    0 & \text{for} \;\; x<0
\end{cases}
$$
and where $\partial^{-1}_x \Lambda^\alpha\rho_0 = K_\alpha*\rho_0\in \Li$ 
(with the function $K_\alpha$ defined at the end of the proof of Lemma \eqref{lem:UEstim})
satisfies
$$
\left(\partial^{-1}_x \Lambda^\alpha\rho_0\right)(\lambda_k x) \to 0 \quad \text{as} \quad\lambda_k\to \infty.
$$
Thus, by the Lebesgue dominated convergence theorem,
\begin{nalign}
        \int_\R u_0^{\lambda_k}(x) \varphi(x,0) \dx 
    =     \int_\R u_0(\lambda_k x) \varphi(x,0) \dx
        \to M_G \int_0^\infty\varphi(u,0) \dy.
    \end{nalign}
   
    Next, choosing $T>0$ such that $\varphi(x,t)=0$ for $(x,t)\in \R\times (T,\infty)$ and using Lemma~\ref{thm:LambdaEstim1} with $p\in(1,\infty)$ and $p' = p/(p-1)$, we have 
    \begin{nalign}
        \left|  \lambda^{-\alpha}_k \int_0^\infty \int_\R  u^{\lambda_k} \rho^{\lambda_k} \Lambda^\alpha \varphi \dy\ds\right| & \\ 
        \leqslant |\lambda_k|^{-\alpha} \|u_0\|_\infty \sup_{s\in[0,T]} \| \Lambda^\alpha \varphi(&\cdot, s) \|_{p/(p-1)} C \int_0^T t^{-1 + \frac{1}{p}} \dt \to 0 \quad \text{as} \quad \lambda_k \to 0. 
    \end{nalign}
    We show analogously 
    \begin{nalign}
         \lambda^{-\alpha}_k \int_0^\infty \int_\R  \rho^{\lambda_k} \Lambda^\alpha \varphi \dy\ds
        \to 0 \quad \text{as} \quad \lambda_k \to 0. 
    \end{nalign}
    Next, for arbitrary $\delta >0$, we decompose (where $\varphi(x,t) = 0$ for $x\in \mathbb{R}$ and $t\geqslant T$)
    \begin{nalign}
        \int_0^\infty \int_\mathbb{R} \rho^\lambda \varphi_s \dy \ds = \int_0^\delta \int_\mathbb{R} \rho^\lambda \varphi_s \dy \ds + \int_\delta^T \int_\mathbb{R} \rho^\lambda \varphi_s \dy \ds. 
    \end{nalign}
    By Lemma \ref{thm:LambdaEstim1} with some $p\in (1,\infty)$ and $p' = p/(p-1)$, we have
    \begin{nalign}
        \int_0^\delta \int_\mathbb{R} \rho^\lambda \varphi_s \dy \ds &\leqslant C(p, M_\rho) \sup_{s\in[0,\delta]} \| \varphi_s(\cdot,s)\|_{p'} \int_0^\delta s^{-1+\frac{1}{p}} \ds \\ &\leqslant C(p, M_\rho) \sup_{s\in[0,\delta]} \| \varphi_s(\cdot,s)\|_{p'} \delta^\frac{1}{p} \to 0 \quad \text{as} \quad \delta \to 0. 
    \end{nalign}
    On the other hand, by the weak convergence in \eqref{eq:WeakLambdaConvergence2}, we obtain
    \begin{nalign}
        \int_\delta^T \int_\mathbb{R}  \rho^{\lambda_k}  \varphi_s \dy \ds \to \int_\delta^T \int_\mathbb{R}  \overline{\rho}  \varphi_s \dy \ds \quad \text{as} \quad \lambda_k \to \infty.
    \end{nalign}
    Since $\delta>0$ is arbitrary small, these two relations imply
    \begin{nalign}
        \int_0^\infty \int_\mathbb{R}  \rho^{\lambda_k}  \varphi_s \dy \ds \to \int_0^\infty \int_\mathbb{R}  \overline{\rho}  \varphi_s \dy \ds \quad \text{as} \quad \lambda_k \to \infty.        
    \end{nalign}
    Similarly, we may pass to the limit as $\lambda_k \to \infty$ in the following terms of system \eqref{eq:URho2eqLambdaWeak}-\eqref{eq:UrhoGLambdaWeak} 
    \begin{nalign}
        \int_0^\infty \int_\mathbb{R}  G^{\lambda_k}  \varphi \dy \ds &\to \int_0^\infty \int_\mathbb{R}  \overline{G}  \varphi \dy \ds, \\
        \int_0^\infty \int_\mathbb{R}  G^{\lambda_k}  \varphi_s \dy \ds &\to \int_0^\infty \int_\mathbb{R}  \overline{G}  \varphi_s \dy \ds, \\
        \int_0^\infty \int_\mathbb{R}  u^{\lambda_k}  \varphi_s \dy \ds &\to \int_0^\infty \int_\mathbb{R}  \overline{u}  \varphi_s \dy \ds, \\
        \int_0^\infty \int_\mathbb{R}  u^{\lambda_k}  \varphi_s \dy \ds &\to \int_0^\infty \int_\mathbb{R}  \overline{u}  \varphi_y \dy \ds. \\           
    \end{nalign}
    Next, we deal with the integrals 
    \begin{nalign}
        \label{eq:IntLam}
        \int_0^\infty \int_\mathbb{R} \rho^{\lambda_k} u^{\lambda_k} \varphi_y \dy \ds \quad \text{and} \quad \int_0^\infty \int_\mathbb{R} G^{\lambda_k} u^{\lambda_k} \varphi_y \dy \ds
    \end{nalign}
    in the analogous way, decomposing the interval with respect to $t\in [0,\delta]$ and $t\in(\delta, T]$ and using the weak convergence of $\varphi^{\lambda_k}$ and $G^{\lambda_k}$ together with the strong convergence of $u^{\lambda_k}$ from \eqref{eq:WeakLambdaConvergence2}. 
   
    Finally, notice that
    $
       G^{\lambda_k}=   u_y^{\lambda_k} - {\lambda_k}^{-\alpha} \Lambda^\alpha \rho^{\lambda_k}.
    $
    Hence, the passage to the limit $\lambda_k \to \infty$ in the third equation of \eqref{eq:UrhoGLambdaWeak} implies $G^{\lambda_k} \to \overline{G} = \overline{u}_y$ weakly in $L^r\big([t_1, t_2], \Lp{}\big)$. 
    Thus, analogously as in the case of the integrals in \eqref{eq:IntLam}, we obtain that
    \begin{nalign}
        \int_0^\infty \int_\mathbb{R} &u^{\lambda_k} \left( u^{\lambda_k}_y - \lambda^{-\alpha} \Lambda ^\alpha\rho^{\lambda_k }\right) \varphi \dy \ds \\ 
        &= \int_0^\infty \int_\mathbb{R} u^{\lambda_k} G^{\lambda_k} \varphi \dy \ds \to \int_0^\infty \int_\mathbb{R} \overline{u}\overline{u}_y \varphi \dy \ds \quad \text{as} \quad \lambda_k \to \infty.
    \end{nalign}
  
    Summarizing, passing to the limit with $\lambda_k \to \infty$ in formulas \eqref{eq:URho2eqLambdaWeak} and \eqref{eq:UrhoGLambdaWeak}, we obtained
    equations 
    \begin{nalign} \label{u:as:3}
   - \int_0^\infty \int_\R \overline\rho\varphi_s \dy \ds - M_\rho\varphi(0, 0) &-  \int_0^\infty \int_\R \overline\rho \overline u \varphi_y \dy \ds = 0, \\
   - \int_0^\infty \int_\R \overline u \varphi_s \dy \ds  -M_G\int_{0}^\infty  \varphi(y,0) \dy &- \int_0^\infty \int_\R  \overline u \overline u_y \varphi \dy \ds = 0,
    \end{nalign}
    and 
    \begin{nalign}
       - \int_0^\infty \int_\R \overline \rho \varphi_s \dy \ds -M_\rho\varphi(0, 0)   - \int_0^\infty \int_\R \overline\rho \overline u \varphi_y \dy \ds &= 0, \\
       - \int_0^\infty \int_\R \overline G \varphi_s \dy \ds -M_G\varphi(0, 0)  - \int_0^\infty \int_\R \overline G \overline u \varphi_y \dy \ds &= 0, \\
       - \int_0^\infty \int_\R \overline u \varphi_y \dy \ds   - \int_0^\infty \int_\R \overline G \varphi \dy \ds & = 0
    \end{nalign}
    for each $\varphi \in C_c^\infty \big(\mathbb{R} \times [0,\infty)\big)$.
    Notice that these are  the weak formulations of the systems
    \begin{nalign}
        \overline{\rho}_s + (\overline{u}\overline{\rho})_y &= 0, \\
        \overline{u}_s + \overline{u}\overline{u}_y&= 0,
    \end{nalign}
    and 
    \begin{nalign} \label{eq:Limit2eqU}
        \overline{\rho}_s + \left(\overline{u}\overline{\rho} \right)_y &= 0, \\
        \overline{G}_s + \left(\overline{u}\overline{G}\right)_y& = 0, \\
        \overline{u}_y &= \overline{G}.
    \end{nalign}
    with the initial conditions 
    \begin{nalign}\label{eq:Limit2eqU:ini}
        \overline\rho(y,0)  = M_\rho \delta_0,  \qquad \overline{G}(y,0)  = M_G \delta_0, \qquad \overline{u}(y,0) = \begin{cases}
            0, & y\leqslant 0, \\
            M_G, & y>  0,
        \end{cases}
    \end{nalign}
    where $\delta_0$ denoted the usual Dirac measure. Below, we will find an explicit unique solutions of these initial value problems. 
  
    Indeed,  note  that $\overline{u} \in L^\infty\big(\mathbb{R} \times (0,\infty)\big)$ and $\overline{u}_y =\overline{G}\in L^r\big([t_1, t_2], \Lp{}\big)$.
    Hence, by \cite[Cor.~8.10]{Brezis}, function $\overline{u}^2\in L^r\big([t_1,t_2], W^{1,p}(\Omega)\big)$ and, moreover,  for each $\varphi \in C_c^\infty$ 
    \begin{align}
        \int_0^\infty \int_\R \overline{u}\overline{u}_y \varphi  \dy \ds = -\int_0^\infty\int_\R \frac{\overline{u}^2}{2} \varphi_y \dy \ds.
    \end{align} 
    Thus, the second equation in \eqref{u:as:3} means $\overline{u}\in L^\infty (\R\times [0, \infty)) $ is a weak solution of the initial value problem for the iniviscid Burgers equation 
    \begin{nalign}
        \label{eq:ULimForm}
        \overline{u}_s + \frac{(\overline{u}^2)_y}{2}&= 0, \qquad \overline{u}_0(y) = 
        \begin{cases}
            0, &  y \leqslant 0, \\ 
            M_G, & y > 0
        \end{cases}
    \end{nalign}
    with $M_G = \int_\mathbb{R} G_0 (x) \dx > 0$. 
    By the estimate of $G$ in \eqref{thm:DecayEstiamtes} with $p=\infty$, we obtain $\Ni{{G}(\cdot,t)} \leqslant C t^{-1}$ for all $t>0$ and this estimate is preserved by the rescaling (Lemma~\ref{thm:LambdaEstim1}) as well as it is kept after passing to the week limit in \eqref{eq:WeakLambdaConvergence}.
    Since, $\overline{G}=\overline{u}_y$, we obtain  the Oleinik condition
    \begin{nalign}
        \Ni{\overline{u}_y(\cdot,s)} = \Ni{\overline{G}(\cdot,s)} \leqslant C s^{-1}\quad \text{for all}\quad s>0,
    \end{nalign}
    which implies (see {\it e.g.} \cite[Ch.~3.4.4]{MR2597943}) that  the function $\overline{u}$ is a unique entropy solution of problem~\eqref{eq:ULimForm} given by the explicit formula
    \begin{nalign}
        \label{eq:URarefaction}
        \overline{u}(y,s) = \begin{cases}
            0, & y\leqslant 0, \\
            \frac{y}{s}, & 0<y \leqslant M_G s, \\
            M_G, & y > M_G s,
        \end{cases}
    \end{nalign}
    which is called as a  {\it rarefaction wave}. 

 Using the relation $\overline{G}=\overline{u}_y$, we calculate
    \begin{nalign}
        \overline{G}(y,s) = 
        \begin{cases}
            0, & y\leqslant 0, \\
            \frac{1}{s}, & 0<y\leqslant M_G s, \\
            0, & y> M_G s.
        \end{cases}
    \end{nalign}

    Note that, for $M_\rho=M_G$, we obtain $\rho(x,t)=G(x,t)$ for almost all $(x,t)\in \R\times [0,\infty)$ because
    both functions satisfy the same linear initial value problems:  equations \eqref{eq:Limit2eqU}   supplemented with  initial conditions \eqref{eq:Limit2eqU:ini} and with 
         $\overline{u}$ given by  formula \eqref{eq:URarefaction}.
        For arbitrary $M_\rho>0$ and $M_G>0$, since the equations are linear, we put
        \begin{nalign}
        \overline{\rho}(y,s) = \frac{M_\rho}{M_G} \overline G(y,s)= \frac{M_\rho}{M_G}
        \begin{cases}
            0, & y\leqslant 0, \\
            \frac{1}{s}, & 0<y\leqslant M_G s, \\
            0, & y> M_G s.
        \end{cases}
    \end{nalign}
\end{proof}


\bibliographystyle{siam}
\bibliography{Library}

\printindex


\end{document}